\documentclass[12pt,reqno,a4paper]{amsart}
\usepackage{enumerate}
\usepackage[totalheight=650pt]{geometry}

\usepackage{amsmath, amssymb}
\usepackage{amsthm}

\usepackage[dvips]{graphicx}                      
\usepackage{psfrag}

\newcommand{\color}[2][{}]{}         

\usepackage{bbm}         

\usepackage{mathrsfs}    
\renewcommand\mathcal\mathscr

\usepackage{accents}     

\newcommand{\hm}[1]{\textbf{*}\leavevmode{\marginpar{\tiny%
$\hbox to 0mm{\hspace*{-0.5mm}$\leftarrow$\hss}%
\vcenter{\vrule depth 0.1mm height 0.1mm width \the\marginparwidth}%
\hbox to 0mm{\hss$\rightarrow$\hspace*{-0.5mm}}$\\\relax\raggedright #1}}}

\numberwithin{equation}{section}

\theoremstyle{plain}            
\newtheorem{theorem}{Theorem}[section]

\newtheorem{proposition}[theorem]{Proposition}
\newtheorem{lemma}[theorem]{Lemma}
\newtheorem{corollary}[theorem]{Corollary}

\theoremstyle{definition}       
\newtheorem{definition}[theorem]{Definition}

\newtheorem{assumption}[theorem]{Assumption}

\theoremstyle{remark}
\newtheorem{remark}[theorem]{Remark}

\newcommand{\Sec}[1]{Section~\ref{sec:#1}}

\newcommand{\Subsec}[1]{Subsection~\ref{ssec:#1}}

\newcommand{\Fig}[1]{Figure~\ref{fig:#1}}

\newcommand{\Thm}[1]{Theorem~\ref{thm:#1}}
\newcommand{\Thms}[2]{Theorems~\ref{thm:#1} and~\ref{thm:#2}}

\newcommand{\Lem}[1]{Lemma~\ref{lem:#1}}

\newcommand{\Cor}[1]{Corollary~\ref{cor:#1}}

\newcommand{\Prp}[1]{Proposition~\ref{prp:#1}}
\newcommand{\Prps}[2]{Propositions~\ref{prp:#1} and~\ref{prp:#2}}

\newcommand{\Rem}[1]{Remark~\ref{rem:#1}}

\newcommand{\Ass}[1]{Assumption~\ref{ass:#1}}
\newcommand{\Asss}[2]{Assumptions~\ref{ass:#1} and~\ref{ass:#2}}

\DeclareMathOperator{\dd}    {d\!}

\DeclareMathOperator{\dom}    {dom}

\DeclareMathOperator{\Aut}    {Aut}

\DeclareMathOperator{\leb}    {{\boldsymbol \lambda}}

\DeclareMathOperator{\ran}    {ran}

\DeclareMathOperator{\supp}   {supp}
\DeclareMathOperator{\vol}    {vol}
\DeclareMathOperator{\tr}     {tr}  

\newcommand{\pot}             {\mathrm {pot}}

\newcommand{\comb}            {\mathrm {comb}}
\newcommand{\diff}            {\mathrm {diff}}
\newcommand{\comp}            {\mathrm {comp}}

\newcommand{\specsymb} {\sigma} 

\newcommand{\spec}[2][{}]   {\specsymb_{\mathrm{#1}}(#2)}

\def\XXint#1#2#3{{\setbox0=\hbox{$#1{#2#3}{\int}$}
     \vcenter{\hbox{$#2#3$}}\kern-.5\wd0}}

\newcommand{\R}{\mathbb{R}} 
\newcommand{\C}{\mathbb{C}} 
\newcommand{\N}{\mathbb{N}} 
\newcommand{\Z}{\mathbb{Z}} 
\newcommand{\Torus}{\mathbb{T}} 

\newcommand{\eps}{\varepsilon} 
\renewcommand{\phi}{\varphi}   
\newcommand{\e}{\mathrm e}  
\newcommand{\im}{\mathrm i} 

\newcommand{\wt}{\widetilde}           
\newcommand {\qf}[1]{\mathfrak{#1}}    

\newcommand {\Exp}[2][{}]{\mathbb{E}_{{#1}}{(#2)}} 
\newcommand {\Prob}   {{\mathbb P}}    

\newcommand{\Sobsymb} {\mathsf H}      
\newcommand{\Contsymb} {\mathsf C}     
\newcommand{\Lsymb}    {\mathsf L}     
\newcommand{\lsymb}    {\ell}          
\newcommand{\Lsqrsymb}    {\mathsf L_2}     
\newcommand{\lsqrsymb}    {\ell_2}          

\newcommand{\Cont}[2][{}]{\Contsymb^{#1}({#2})}

\newcommand{\Contc}[2][{}]{\Contsymb_{\mathrm c}^{#1}({#2})}

\newcommand{\Lp}[2][p]{\Lsymb_{#1}({#2})} 

\newcommand{\Lsqr}[2][{}]{\Lsymb_2^{#1}({#2})} 

\newcommand{\lsqr}[2][{}]{\lsymb_2^{#1}({#2})}   

\newcommand{\Linfty}[2][{}]{\Lsymb_\infty^{#1}({#2})} 

\newcommand{\Sob}[2][1]{\Sobsymb^{#1}({#2})}         
\newcommand{\Sobx}[3][1]{\Sobsymb_{{#2}}^{#1}({#3})} 

\newcommand{\abs}[2][{}]{\lvert#2\rvert_{{#1}}}
\newcommand{\abssqr}[2][{}]{\lvert{#2}\rvert^2_{#1}} 
\newcommand{\bigabs}[2][{}]{\bigl\lvert{#2}\bigr\rvert_{#1}}     
\newcommand{\Bigabs}[2][{}]{\Bigl\lvert{#2}\Bigr\rvert_{#1}}     
\newcommand{\Bigabssqr}[2][{}]{\Bigl\lvert{#2}\Bigr\rvert^2_{#1}}
\newcommand{\norm}[2][{}]{\|{#2}\|_{{#1}}}    
\newcommand{\normsqr}[2][{}]{\|{#2}\|^2_{#1}} 

\newcommand{\iprod}[3][{}]{\langle{#2},{#3}\rangle_{#1}}  

\newcommand{\set}[2]{\{ \, #1 \, | \, #2 \, \} }      
\newcommand{\bigset}[2]{\bigl\{ \, #1 \, \bigl|\bigr. \, #2 \, \bigr\} }
\newcommand{\Bigset}[2]{\Bigl\{ \, #1 \, \Bigl|\Bigr. \, #2 \, \Bigr\} }

\newcommand{\map}[3]{ #1 \colon #2 \longrightarrow #3}    

\newcommand{\bd}  {\partial}                
\newcommand{\clo}[1]{\overline{{#1}}} 

\newcommand{\compl}[1]{#1^{\mathrm c}}       
\newcommand{\dcup}{\mathbin{\mathaccent\cdot\cup}}

\DeclareMathOperator*{\bigdcup}{\mathaccent\cdot{\bigcup}}

\newcommand{\1}{\mathbbm 1}                

\newcommand{\Und}{\qquad\text{and}\qquad}

\newcommand{\Dir}{{\mathrm D}}              
\newcommand{\lapl}[2][{}]{\Delta_{{#2}}^{{#1}}}

\newcommand{\de} {\mathord{D}} 

\newcommand{\mc}{\mathcal}
\newcommand{\ul}{\underline}
\newcommand{\orul}[1]{\underrightarrow{#1}}

\newcommand{\RR}{\mathbb{R}}
\newcommand{\CC}{\mathbb{C}}

\newcommand{\ZZ}{\mathbb{Z}}

\begin{document}

\title[Continuity of the IDS on random length metric
graphs]{Continuity of the integrated density of states on random
  length metric graphs}

\author[D.~Lenz]{Daniel Lenz}
\author[N.~Peyerimhoff]{Norbert Peyerimhoff} 
\author[O.~Post]{Olaf Post} 
\author[I.~Veseli\'c]{Ivan Veseli\'c}

\address[D.~Lenz]{Friedrich-Schiller-Universit\"at Jena, Fakult\"at f\"ur Mathematik \& Informatik,
Mathematisches Institut, 07737 Jena, Germany} \email{daniel.lenz@uni-jena.de}
\urladdr{www.tu-chemnitz.de/mathematik/mathematische\_physik/}

\address[N.~Peyerimhoff]{Department of Mathematical Sciences, Durham
  University, Science Laboratories South Road, Durham, DH1 3LE, Great
  Britain} \email{norbert.peyerimhoff@durham.ac.uk}
\urladdr{www.maths.dur.ac.uk/~dma0np/}

\address[O.~Post]{Institut f\"ur Mathematik, SFB 647 ``Space -- Time
  -- Matter'', Humboldt-Universit\"at zu Berlin, Rudower Chaussee~25,
  12489 Berlin, Germany} \email{post@math.hu-berlin.de}
\urladdr{www.math.hu-berlin.de/\protect{\char126}post/}

\address[I.~Veseli\'c]{Fakult\"at f\"ur Mathematik,TU Chemnitz,
  D-09107 Chemnitz, Germany \\ \hspace*{1,3em}\& Emmy-Noether
  Programme of the DFG}
\urladdr{www.tu-chemnitz.de/mathematik/schroedinger/members.php}

\date{\today, \emph{File}: \texttt{\jobname.tex}}


\begin{abstract}
  We establish several properties of the integrated density of states
  for random quantum graphs: Under appropriate ergodicity and
  amenability assumptions, the integrated density of states can be
  defined using an exhaustion procedure by compact subgraphs. A trace
  per unit volume formula holds, similarly as in the Euclidean case.
  Our setting includes periodic graphs.  For a model where the edge
  lengths are random and vary independently in a smooth way we prove a
  Wegner estimate and related regularity results for the integrated
  density of states.

  These results are illustrated for an example based on the Kagome
  lattice.  In the periodic case we characterise all compactly
  supported eigenfunctions and calculate the position and size of
  discontinuities of the integrated density of states.
\end{abstract}

\maketitle

%
\section{Introduction}
\label{sec:intro}
%

Quantum graphs are Laplace or Schr\"odinger operators on metric
graphs. As structures intermediate between discrete and continuum
objects they have received quite some attention in recent years in
mathematics, physics and material sciences, see e.g. the recent
proceeding volume~\cite{ekkst:08} for an overview.

Here, we study periodic and random quantum graphs.  Our results
concern spectral properties which are related to the integrated
density of states (IDS), sometimes called spectral distribution
function.  As in the case of random Schr\"odinger operators in
Euclidean space, disorder may enter the operator via the potential.
Moreover, and this is specific to quantum graphs, randomness may also
influence the characteristic geometric ingredients determining the
operator, viz.
\begin{itemize}
 \item the lengths of the edges of the metric graph and
 \item the vertex conditions at each junction between the edges.
\end{itemize}

In the present paper we pay special attention to randomness in these
geometric data.  Our results may be summarised as follows. For quite
wide classes of quantum graphs we establish
\begin{itemize}
\item the existence, respectively the convergence in the macroscopic
  limit, of the integrated density of states under suitable
  ergodicity and amenability conditions (see \Thm{ids}),
\item a trace per unit volume formula for the IDS
  (see equation~\eqref{eq:idsn}),
\item a Wegner estimate for random edge length models (assuming
  independence and smoothness for the disorder) (\Thm{wegner}). This
  implies quantitative continuity estimates for the IDS
  (\Cor{wegner}).
\end{itemize}

These abstract results are illustrated by the thorough discussion of
an example concerning a combinatorial and a metric graph based on the
Kagome lattice.  In this case we calculate positions and sizes of all
jumps of the IDS. Our results show the effect of smoothing of the IDS
via randomness.

The article is organised as follows: In the remainder of this section
we summarise the origin of results about the construction of the IDS
and of Wegner estimates and point out aspects of the proofs which are
different in the case of quantum graphs in comparison to random
Schr\"odinger operators on $L^2(\RR^d)$ or $\ell^2(\ZZ^d)$.  We
mention briefly recent results about spectral properties of random
quantum graphs which are in some sense complementary to ours. Finally,
we point out some open problems in this field of research.  In the
next section, we introduce the random length model and state the main
results. In \Sec{example} we present the Kagome lattice example. In
\Sec{approx} we prove \Thm{ids} concerning the approximability of the
IDS. Finally, in \Sec{wegner} we prove the Wegner estimate
\Thm{wegner}.

Intuitively, the IDS concerns the number of quantum states per unit
volume below a prescribed energy.  From the physics point of view the
natural definition of this quantity is via a macroscopic limit. This
amounts to approximating the (ensemble-averaged) spectral distribution
function of an operator on the whole space by normalised eigenvalue
counting functions associated to finite-volume restrictions of the
operator.  For ergodic random and almost-periodic operators in
Euclidean space this approach has been implemented rigorously
in~\cite{Pastur-71,Shubin-79}, and developed further in a number of
papers, among them~\cite{KirschM-82c},~\cite{Matsumoto-93}
and~\cite{HupferLMW-01b}.  All these operators were stationary and
ergodic with respect to a commutative group of translations.  For
graphs and manifolds beyond Euclidean space the relevant group is in
general no longer abelian.  The first result establishing the
approximability of the IDS of a periodic Schr\"odinger operator on a
manifold was~\cite{adachi-sunada:93}.  An important assumption on the
underlying geometry is amenability.  Analogous results on transitive
graphs have been established e.g.~in~\cite {mathai-yates:02}
and~\cite{msy:03}.  For Schr\"odinger operators with a random
potential on a manifold with an amenable covering group the existence
of the IDS was established in \cite{peyerimhoff-veselic:02}, and for
Laplace-Beltrami operators with random metrics in~\cite{lpv:04}.  For
analogous results for discrete operators on amenable graphs see
e.g.~\cite{Veselic-05b} and~\cite {LenzV}.  A key ingredient of the
proofs of the above results is the amenable ergodic theorem
of~\cite{lindenstrauss:01}.  More recenty, the question of
approximation of the IDS uniformly with respect to the energy variable
has been pursued, see for instance~\cite{lmv:08} and the references
therein.

Independently of the approximability by finite volume eigenvalue
counting functions it is possible to give an abstract definition of
the IDS by an averaged trace per unit volume formula,
see~\cite{Shubin-79,BellissardLT-85,lenz:99,LenzPV-07}.  In the
amenable setting, both definitions of the IDS coincide.

For a certain class of metric graphs the approximability of the IDS
has been established before.
In~\cite{helm-veselic:07,glv:07,glv.in:08} this has been carried out
for random metric graphs with an $\ZZ^d$ structure. A step of the
proof which is specific to the setting of quantum graphs concerns the
influence of finite rank perturbations on eigenvalue counting
functions.  When one considers Laplacians on manifolds, one would
rather use the principle of not feeling the boundary of heat kernels,
cf.~e.g.~\cite{adachi-sunada:93,peyerimhoff-veselic:02,lpv:04}, to
derive the analogous step of the proof.

Next we discuss the literature on Wegner estimates and on the
regularity of the IDS.  Wegner gave in~\cite{Wegner-81} convincing
arguments for the Lipschitz continuity of the IDS of the discrete
Anderson model on $\ell^2(\ZZ^d)$. The proof is based on an estimate
for the expected number of eigenvalues in a finite energy interval of
a restricted box Hamiltonian.  A rigorous proof of the latter estimate
was given in~\cite{kirsch:96} (for the analogous alloy-type model on
$L^2(\RR^d)$).  However, the bound of~\cite{kirsch:96} was not
sufficient to establish the Lipschitz continuity of the IDS.
In~\cite{chn:01} tools to prove H\"older continuity were supplied, see
also~\cite{hknsv:06}.  They concern bounds on the spectral shift
function. Up to now the most widely applicable result concerning the
Lipschitz-continuity of the IDS is given in~\cite{chk:07}.  An
alternative approach to derive Lipschitz continuity of the IDS goes
via spectral averaging of resolvents,
see~\cite{kotani-simon:87,combes-hislop:94}.  However, this method
requires more assumptions on the underlying model.

Wegner's estimate and all references mentioned so far concern the case
where the random variables couple to a perturbation which is a
non-negative operator. If this is not the case, additional ideas are
necessary to obtain the desired bounds,
see~\cite{klopp:95a,veselic:02a, hislop-klopp:02,kostrykin-veselic:06,
  veselic:pre08}.  In our situation, where the perturbation concerns
the metric of the underlying space, the dependence on the random
variables is not monotone. This is also the case for random metrics on
manifolds studied in~\cite{lppv:08}.  To deal with non-monotonicity,
the proof of the Wegner estimate (Theorem \ref{thm:wegner}) takes up
an idea developed in~\cite{lppv:08}, which is not unrelated
to~\cite{klopp:95a}.  The relevant formula used in the proof is
\eqref{eq:lambdai.scal}.  We need also a partial integration formula
whose usefulness was first seen in~\cite{hislop-klopp:02}.

In the context of quantum graphs it is not necessary to rely on
sophisticated estimates on the spectral shift function.  It is
sufficient to adapt a finite rank perturbation bound, which was used
in~\cite{kirsch-veselic:02b} for the analysis of one-dimensional
random Schr\"odinger operators.  These estimates are closely related
to the finite rank estimates mentioned earlier in the context of the
approximability of the IDS.  For Schr\"odinger operators on metric
graphs where the randomness enters via the potential, Wegner estimates
have been proved in~\cite{helm-veselic:07,GruberV-08,ghv.in:08}.  In
the recent preprint~\cite{klopp-pankrashkin:09} a Wegner estimate for
a model with $\Z^d$-structure and random edge lengths has been
established.  The proof is based on different methods than we use in
the present paper.

Next we want to explain an application of Wegner estimates apart from
the continuity of the IDS. It concerns the phenomenon of localisation
of waves in random media. More precisely, for certain types of random
Schr\"odinger operators on $\ell^2(\ZZ^d)$ and on $L^2(\RR^d)$ it is
well known that in certain energy intervals near spectral boundaries
the spectrum is pure point.  There are two basic methods to establish
this fact (apart form the one-dimensional situation where specific
methods apply).  The first one is called multiscale analysis and was
invented in~\cite{froehlich-spencer:83}. The second approach from
\cite{aizenman-molchanov:93} is called fractional moment method or
Aizenman-Molchanov method.  A certain step of the localisation proof
via multiscale analysis concerns the control of spectral resonances of
finite box Hamiltonians.  A possibility to achieve this control is the
use of a Wegner estimate. In fact, the Wegner estimates needed for
this purpose are much weaker than those necessary to establish
regularity of the IDS. This has been discussed in the context of
random quantum graphs in Section 3.2 of~\cite{ghv.in:08}.

Recently localisation has been proven for several types of random
quantum graphs.
In~\cite{ehs:07,klopp-pankrashkin:08,klopp-pankrashkin:09} this has
been done for models with $\ZZ^d$-structure,
while~\cite{hislop-post:pre06} considers operators on tree-graphs.  On
the other hand, delocalisation, i.e.~existence of absolutely
continuous spectrum, for quantum graph models on trees has been shown
in~\cite{asw:06a}. This result should be seen in the context of
earlier, similar results for combinatorial tree
graphs~\cite{klein:96,klein:98,asw:06a,asw:06d,fhs:06}.

Now let us discuss some open questions concerning random quantum graph
models.  As for models on $\ell^2(\ZZ^d)$, proofs of localisation
require that the random variables entering the operators should have a
regular distribution. In particular, if the law of the variables is a
Bernoulli measure, no known proof of localisation applies. This is
different for random Schr\"odinger operators $L^2(\RR^d)$.  Using a
quantitative version of the unique continuation principle for
solutions of Schr\"odinger equations, localisation was established
in~\cite{BourgainK-05} for certain models with Bernoulli disorder.
The proof does not carry over to the analogous model on
$\ell^2(\ZZ^d)$, since there is no appropriate version of the unique
continuation principle available.  For random quantum graphs the
situation is even worse, since they exhibit in great generality
compactly supported eigenfunctions, even if the underlying graph is
$\ZZ^d$.

Like for random, ergodic Schr\"odinger operators on $\ell^2(\ZZ^d)$
and on $L^2(\RR^d)$ there is no proof of delocalisation for random
quantum graphs with $\ZZ^d$ structure. In the above mentioned papers
on delocalisation it was essential that the underlying graph is a
tree.  An even harder question concerns the mobility edge. Based on
physical reasoning one expects that localised point spectrum and
delocalised absolutely continuous spectrum should be separated in
disjoint intervals by mobility edges. In the context of random
operators where the disorder enters via the geometry this leads to an
intriguing question pointed out already in~\cite{ChayesCFST-86}.  If
one considers a graph over $\ZZ^d$ which is diluted by a percolation
process, the Laplacian on the resulting combinatorial or metric graph
has a discontinuous IDS.  In fact, the set of jumps can be
characterised rather explicitely and is dense in the
spectrum~\cite{ChayesCFST-86,Veselic-05b,glv.in:08}.  Now the question
is, where the eigenvalues of these strongly localised states repell in
some manner absolutely continuous spectrum (if it exists at all).

For the interested reader we provide here references to textbook
accounts of the issues discussed above. They concern the more
classical models on $L^2(\RR^d)$ and $\ell^2(\ZZ^d)$, rather than
quantum graphs.  In~\cite{veselic:07} one can find a detailed
discussion and proofs of the approximability of the IDS by its finite
volume analogues and of Wegner estimates.  The survey article
\cite{KirschM-07} is devoted to the IDS in general, while the
multiscale proof of localisation is exposed in the monograph
\cite{stollmann:01}.  The theory of random Schr\"odinger operators is
presented from a broader perspective in the books
\cite{carmona-lacroix:90, pastur-figotin:92} and in the summer school
notes~\cite{Kirsch-89a,kirsch:pre07}.

\subsection*{Acknowledgements}

The second author is grateful for the kind invitation to the
Humboldt University of Berlin which was supported by the SFB~647.
NP and OP also acknowledge the financial support of the Technical
University Chemnitz.

\section{Basic notions, model and results}
\label{sec:qg}

In the following subsections, we fix basic notions (metric graphs,
Laplacians and Schr\"odinger operators with vertex conditions),
introduce the random length model and state our main results.  For
general treatments and further references on metric graphs, we refer
to~\cite{ekkst:08}.

\subsection{Metric graphs}

Since our random model concerns a perturbation of the \emph{metric}
structure of a graph, we carefully distinguish between
\emph{combinatorial}, \emph{topological} and \emph{metric graphs}.  A
\emph{combinatorial graph} $G = (V,E,\bd )$ is given by a countable
vertex set $V$, a countable set $E$ of edge labels and a map $\bd(e) =
\{ v_1,v_2 \}$ from the edge labels to (unordered) pairs of vertices.
If $v_1 = v_2$, we call $e$ a \emph{loop}. Note that this definition
allows multiple edges, but we only consider locally finite
combinatorial graphs, i.e., every vertex has only finitely many
adjacent edges.  A topological graph $X$ is a topological model of a
combinatorial graph together with a choice of directions on the edges:
\begin{definition}
  \label{def:tg}
  A \emph{(directed) topological graph} is a CW-complex $X$ containing
  only (countably many) $0$- and $1$-cells. The set $V=V(X) \subset X$
  of $0$-cells is called the \emph{set of vertices}.  The $1$-cells of
  $X$ are called \emph{(topological) edges} and are labeled by the
  elements of $E=E(X)$ (the \emph{(combinatorial) edges}), i.e., for
  every edge $e \in E$, there is a continuous map
  $\map{\Phi_e}{[0,1]}X$ whose image is the corresponding (closed)
  $1$-cell, and $\map{\Phi_e}{(0,1)}{\Phi_e((0,1))} \subset X$ is a
  homeomorphism. A $1$-cell is called a \emph{loop} if
  $\Phi_e(0)=\Phi_e(1)$. The map $\map{\bd = (\bd_-,\bd_+)} E {V
    \times V}$ describes the direction of the edges and is defined by
  \begin{equation*}
    \bd_-e := \Phi_e(0) \in V, \qquad
    \bd_+e := \Phi_e(1) \in V.
  \end{equation*}
  For $v \in V$ we define
  \begin{equation*}
    E_v^\pm=E_v^\pm(X) := \set{e \in E}{\bd_\pm e=v}.
  \end{equation*}
  The set of all adjacent edges is defined as the \emph{disjoint}
  union\footnote{The disjoint union is necessary in order to obtain
    \emph{two} different labels in $E_v(X)$ for a loop.}
  \begin{equation*}
    E_v=E_v(X) := E_v^+(X) \dcup E_v^-(X).
  \end{equation*}
  The \emph{degree} of a vertex $v \in V$ in $X$ is defined as
  \begin{equation*}
    \deg v = \deg_X(v) := \abs{E_v}
                        = \abs{E_v^+} + \abs{E_v^-}.
  \end{equation*}
  A \emph{topological subgraph $\Lambda$} is a CW-subcomplex of
  $X$, and therefore $\Lambda$ is itself a topological graph with
  (possible empty) boundary $\bd \Lambda := \Lambda \cap \clo{\compl
  \Lambda} \subset V(X)$.
\end{definition}

Since a topological graph is a topological space, we can introduce the
space $\Cont X$ of $\C$-valued continuous functions and the associated
notion of measurability. A metric graph is a topological graph where
we assign a length to every edge.

\begin{definition}
  \label{def:mg}
  A \emph{(directed) metric} graph $(X,\ell)$ is a topological
  graph $X$ together with a \emph{length function}
  $\map \ell {E(X)}{(0,\infty)}$. The length function induces an
  identification of the interval $I_e:=[0,\ell(e)]$ with the edge
  $\Phi_e([0,1])$ (up to the end-points of the corresponding $1$-cell,
  which may be identified in $X$ if $e$ is a loop) via the map
  \begin{equation*}
    \map{\Psi_e}{I_e} X, \qquad
    \Psi_e(x) = \Phi_e\Bigl(\frac x {\ell(e)} \Bigr).
  \end{equation*}
\end{definition}

Note that every topological graph $X$ can be canonically regarded as a
metric graph where \emph{all edges have length one}. The
corresponding length function $\1_{E(X)}$ is denoted by $\ell_0$. In
our random model, we will consider a fixed topological graph $X$ with
a random perturbation $\ell_\omega$ of this length function $\ell_0$.

To simplify matters, we canonically identify a metric graph $(X,\ell)$
with the disjoint union $X_\ell$ of the intervals $I_e$ for all $e \in
E$ subject to appropriate identifications of the end-points of these
intervals (according to the combinatorial structure of the graph),
namely
\begin{equation*}
  X_\ell := \bigdcup_{e \in E} I_e / {\sim}.
\end{equation*}
The coordinate maps $\{\Psi_e\}_e$ can be glued together to a map
\begin{equation}
  \label{eq:GX}
  \map {\Psi_\ell} {X_\ell} X.
\end{equation}
\begin{remark}
  \label{rem:met.space}

  A metric graph is canonnically equpped with a metric and a measure.
  Given the information about the lenght of edges, each path in
  $X_\ell$ has a well defined lenght. The distance between two
  arbitrary points $x,y \in X_\ell$ is defined as the infimum of the
  lenghts of paths joining the two points. The measure on $X_\ell$ is
  defined in the following way. For each measurable $\Lambda \subset
  X$ the sets $\Lambda \cap \psi_e(I_e)$ are measurable as well, and
  are assigned the Lebesgue measure of the preimage
  $\psi_e^{-1}(\Lambda \cap \psi_e(I_e))$.  Consequently, we define
  the volume of $\Lambda$ by
  \begin{equation}
    \label{eq:voldef}
\vol(\Lambda,\ell) :=\sum_{e \in E} \leb
 \bigl(\psi_e^{-1}(\Lambda \cap \psi_e(I_e)) \bigr)
  \end{equation}
\end{remark}

Using the identification~\eqref{eq:GX}, we define the function space
$\Lsqr {X,\ell}$ as
\begin{gather*}
  \Lsqr{X,\ell} := \bigoplus_{e \in E} \Lsqr {I_e}, \qquad
  f = \{f_e\}_e \quad \text{with $f_e \in \Lsqr{I_e}$ and}\\
  \normsqr[\Lsqr{X,\ell}] f
  = \sum_{e \in E} \int_{I_e} \abs{f_e(x)}^2 \dd x.
\end{gather*}

\subsection{Operators and vertex conditions}

For a given metric graph $(X,\ell)$, we introduce the operator
\[ (\de f)_e(x) = (\de_\ell f)_e(x) = \frac{d f_e}{dx} (x), \] where
the derivative is taken in the interval $I_e = [0,\ell(e)]$.  Note
that both the norm in $\Lsqr{X,\ell}$ and $\de =\de_\ell$ depend on
the length function. This observation is particularly important in our
random length model below, where we perturb the canonical length
function $\ell_0 =\1_{E(X)}$ and therefore have (a priori) different
spaces on which a function $f$ lives. Our point of view is that $f$ is
a function on the fixed underlying topological graph $X$, and that the
metric spaces are canonically identified via the maps $\map
{\Psi_{\ell_0}^{-1} \circ \Psi_\ell} {(X,\ell)} {(X,\ell_0)}$. One
easily checks that
\begin{subequations}
  \begin{gather}
    \label{eq:L2.resc}
    \normsqr[\Lsqr{X,\ell}] f
    = \sum_{e \in E} \ell(e) \int_{(0,1)} \abs{f_e(x)}^2 \dd x,\\
    \label{eq:D.resc}
    (\de_\ell f)_e(x)
    = \frac 1{\ell(e)}  (\de_{\ell_0} f)_e
           \Bigl( \frac 1{\ell(e)} x \Bigr),
  \end{gather}
\end{subequations}
where $f_e$ and $\de_{\ell_0} f$ on the right side are considered as
functions on $[0,1]$ via the identification $\Psi_{\ell_0}^{-1} \circ
\Psi_\ell$.

Next we introduce general vertex conditions for Laplacians $\lapl
{(X,\ell)} = - {\de_\ell}^2$ and Schr\"odinger operators $H_{(X,\ell)}
= \lapl {(X,\ell)} + q$ with real-valued potentials $q \in \Linfty X$.
The \emph{maximal} or \emph{decoupled} Sobolev space of order $k$ on
$(X,\ell)$ is defined by
\begin{gather*}
  \Sobx[k] \max {X,\ell}
  := \bigoplus_{e \in E} \Sob [k]{I_e}\\
  \normsqr[{\Sobx[k] \max {X,\ell}}] f
  := \sum_{e \in E} \normsqr[{\Sob [k]{I_e}}] {f_e}.
\end{gather*}
Note that $\map{\de_\ell} {\Sobx[k+1] \max {X,\ell}} {\Sobx[k] \max
  {X,\ell}}$ is a bounded operator. We introduce the following two different
evaluation maps ${\Sobx[1] \max {X,\ell}} \longrightarrow
{\bigoplus_{v \in V} \C^{E_v}}$:
\begin{equation*}
  {\ul f}_{\, e}(v) :=
  \begin{cases}
    f_e(0), & \text{if $v=\bd_-e$,}\\
    f_e(\ell(e)), & \text{if $v=\bd_+e$,}
  \end{cases} \quad\text{and}\quad
  \orul f_{e}(v) :=
  \begin{cases}
    -f_e(0), & \text{if $v=\bd_-e$,}\\
    f_e(\ell(e)), & \text{if $v=\bd_+e$,}
  \end{cases}
\end{equation*}
and $\ul f(v) = \{ {\ul f}_{\, e}(v) \}_{e \in E_v} \in \C^{E_v}$,
${\orul f}(v) = \{ {\orul f}_{e}(v) \}_{e \in E_v} \in \C^{E_v}$. It
follows from standard Sobolev estimates (see
e.g.~\cite[Lem.~8]{kuchment:04}) that these evaluation maps are
bounded by $\max\{(2/\ell_{\min})^{1/2},1\}$, provided the minimal
edge length
\begin{equation}
  \label{eq:len.bdd}
  0<\ell_{\min} 
  := \inf_{e \in E} \ell(e)
\end{equation}
is strictly positive. The second evaluation map is used in connection
with the derivative $\de f$ of a function $f \in \Sobx[2] \max
{X,\ell}$. Note that $\orul {\de f}$ is independent of the orientation
of the edge.

A \emph{single-vertex condition} at $v \in V$ is given by a Lagrangian
subspace $L(v)$ of the Hermitian symplectic vector space $(\C^{E_v}
\oplus \C^{E_v},\eta_v)$ with canonical two-form $\eta_v$ defined by
\begin{equation*}
  \eta_v((x,x'),(y,y')) := \iprod {x'} y - \iprod x {y'},
\end{equation*}
where $\iprod \cdot \cdot$ denotes the standard unitary inner product
in $\C^{E_v}$. The set of all Lagrangian subspaces of $(\C^{E_v}
\oplus \C^{E_v},\eta_v)$ is denoted by ${\mathcal L}_v$ and has a
natural manifold structure (see,
e.g.,~\cite{harmer:00,kostrykin-schrader:99} for more details on these
notions).  A Lagrangian subspace $L(v)$ can uniquely be described by
the pair $(Q(v),R(v))$ where $Q(v)$ is an orthogonal projection in
$\C^{E_v}$ with range $\mc G(v) := \ran Q(v)$ and $R(v)$ is a
symmetric operator on $\mc G(v)$ such that
\begin{equation}
  \label{eq:qg.vx.cond}
  L(v) := \bigset {(x,x')}
    {(1-Q(v))x = 0, \quad Q(v) x' = R(v) x}
\end{equation}
(see e.g.~\cite{kuchment:04}).

A field of single-vertex conditions $L := \{L(v)\}_{v \in V}$ is
called a \emph{vertex condition}. We say that $L$ is \emph{bounded},
if
\begin{equation}
  \label{eq:vx.pot.bdd}
  C_R := \sup_{v \in V} \norm {R(v)} < \infty,
\end{equation}
where the norm is the operator norm on $\mc G(v)$.  For any such
bounded vertex condition $L$, a bounded potential $q$ and a metric
graph $(X,\ell)$ with $\ell_{\min}>0$, we obtain a self-adjoint
Schr\"odinger operator $H_{(X,\ell),L} = \lapl {(X,\ell),L} + q$, by
choosing the domain
\begin{equation*}
  \dom H_{(X,\ell),L} := 
    \set{f \in \Sobx[2] \max {X,\ell}}
     { (\ul f(v),\orul {\de f}(v)) \in L(v) 
             \text{ for all $v \in V$}}.
\end{equation*}
Of particular interest are the following vertex conditions with
vanishing vertex operator $R(v) = 0$ for all $v \in V$:
\emph{Dirichlet} vertex conditions (where $L(v) = \{ 0 \} \oplus
\C^{E_v}$ or $\mc G(v)= \{0\}$), \emph{Kirchhoff} (also known as
\emph{free}) vertex conditions (where $(x,x') \in L(v)$ if all
components of $x$ are equal and the sum of all components of $x'$ add
up to zero, or equivalently $\mc G(v) = \C(1,\dots,1)$) and
\emph{Neumann} vertex conditions (where $L(v) = \C^{E_v} \oplus \{ 0
\}$ or equivalently $\mc G(v) = \C^{E_v}$).

\subsection{Random length model}
\label{ssec:ranlengthmod}

The underlying \emph{geometric} structure of a random length model is
a random length metric graph. A \emph{random length metric graph} is
based on a fixed topological graph $X$ with $V$ and $E$ the sets of
vertices and edges of $X$, a probability space $(\Omega, \Prob)$, and
a \emph{measurable} map $\map \ell {\Omega \times E} {(0,\infty)}$,
which describes the random dependence of the edge lengths. We also
assume that there are $\omega$-independent constants
$\ell_{\min},\ell_{\max} > 0$ such that $\ell_{\min} \le
\ell_\omega(e) \le \ell_{\max}$ for all $\omega \in \Omega$ and $e \in
E$. We will use the notation $\ell_\omega(e) := \ell(\omega,e)$.

A \emph{random length model} associates to such a geometric structure
$(X,\Omega,\Prob,\ell)$ a random family of Schr\"odinger operators
$H_\omega$, by additionally introducing \emph{measurable} maps $\map
{L(v)} \Omega {\mc L_v}$ for all $v \in V$, and $\map q {\Omega \times
  X} \R$, describing the random dependence of the vertex conditions
and the potentials of these operators. We will use the notation
$L_\omega := \{L_\omega(v)\}_{v \in V}$ and $q_\omega(x) =
q(\omega,x)$.  We assume that we have constants $C_R, C_\pot>0$ such
that
\begin{equation}
  \label{eq:ran.pot.bdd}
  \norm[\infty]{q_\omega} \le C_\pot
  \Und
  \norm{R_\omega(v)} \le C_R
\end{equation}
for almost all $\omega \in \Omega$ and all $v \in V$, where $R_\omega(v)$
is the vertex operator associated to $L_\omega(v)$.
From~\eqref{eq:ran.pot.bdd} and the lower length
bound~\eqref{eq:len.bdd} it follows that the Schr\"odinger operators
$H_\omega := \lapl{\omega} + q_\omega$ are self-adjoint and bounded
from below by some constant $\lambda_0 \in \R$ uniformly in $\omega
\in \Omega$ (see \Lem{lower.bd}).  We call the tuple
$(X,\Omega,\Prob,\ell,L,q)$ a \emph{random length model} with
associated Laplacians and Schr\"odinger operators $\lapl{\omega}$ and
$H_\omega$ and underlying random metric graphs $(X,\ell_\omega)$.

\subsection{Approximation of the IDS via exhaustions}

Let us describe the setting, for which our first main result holds.

\begin{assumption}
  \label{ass:covgraph}
  Let $(X,\Omega,\Prob,\ell,L,q)$ be a random length model with the
  following properties:
  \begin{enumerate}
  \item The topological graph $X$ is {non-compact} and connected
    with underlying (undirected) combinatorial graph $G = (V,E,\bd)$.
    There is a subgroup $\Gamma \subset \Aut(G)$, acting freely on $V$
    with only finitely many orbits. Then $\Gamma$ acts also
    canonically on $X$ (but does not necessarily respect the
    directions) by
    \begin{equation*}
      \gamma \Phi_e(x) = 
      \begin{cases}
        \Phi_{\gamma e}(x) & \text{if $\bd_\pm(\gamma e) =
          \gamma(\bd_\pm e)$,}\\
        \Phi_{\gamma e}(1-x) & \text{if $\bd_\pm(\gamma e) =
          \gamma(\bd_\mp e)$.}
      \end{cases}
    \end{equation*}
    This action carries over to $\Gamma$-actions on the metric graphs
    $(X,\ell_0)$ and $(X,\ell_\omega)$ via the
    identification~\eqref{eq:GX}.  Note that $\Gamma$ acts even
    \emph{isometrically} on the equilateral graph $(X,\ell_0)$ with
    $\ell_0 = \1_E$. We can think of $(X,\ell_0)$ as a \emph{covering}
    of the compact topological graph $(X/\Gamma,\ell_0)$.

  \item We also assume that $\Gamma$ acts \emph{ergodically} on
    $(\Omega,\Prob)$ by measure preserving transformations with the
    following consistencies between the two $\Gamma$-actions on $X$
    and $\Omega$:
    \begin{subequations}
      \begin{description}
      \item[Metric consistency] We assume that
        \begin{equation}
          \label{eq:ell.cons}
          \ell_{\gamma \omega}(e) = \ell_\omega(\gamma e)
        \end{equation}
        for all $\gamma \in \Gamma$, $\omega \in \Omega$ and $e \in
        E$.  This implies that for every $\gamma \in \Gamma$, the map
        \begin{equation*}
          \map \gamma {(X,\ell_\omega)} {(X,\ell_{\gamma \omega})}
        \end{equation*}
        is an isometry between two (different) metric graphs.
        Moreover, the induced operators
        \begin{equation*}
          \map {U_{(\omega,\gamma)}} {\Lsqr {X,\ell_{\gamma^{-1} \omega)}}} 
          {\Lsqr {X,\ell_\omega}}
        \end{equation*}
        are unitary.
      \item[Operator consistency] The transformation behaviour of
        $q_\omega$ and $L_\omega$ is such that we have for all $\omega
        \in \Omega$, $\gamma \in \Gamma$,
        \begin{equation}
          \label{eq:opcons}
          H_\omega = U_{(\omega.\gamma)}
          H_{\gamma^{-1} \omega} U_{(\omega, \gamma)}^*.
        \end{equation}
      \end{description}
    \end{subequations}
    Such a random length model $(X,\Omega,\Prob,\ell,L,q)$ is called a
    \emph{random length covering model} with associated operators
    $H_\omega$ and covering group $\Gamma$.
  \end{enumerate}
\end{assumption}

\begin{remark}
  \label{rem:trivrandfam}
  The simplest random length covering model is given when the
  probability space $\Omega$ consists of only one element with
  probability $1$. In this case, we have only one length function
  $\ell = \ell_\omega$, one vertex condition $L = L_\omega$, and one
  potential $q = q_\omega$.  The corresponding family of operators
  consists then of a single operator $H = H_\omega$. Moreover, the
  metric consistency means that $\Gamma$ acts isometrically on
  $(X,\ell)$, and the operator consistency is nothing but the
  periodicity of $H$, i.e., the property that $H$ commutes with the
  induced unitary $\Gamma$-action on $\Lsqr {X,\ell}$.
\end{remark}

Next, we introduce some more notation. Let $\mc F_0$ be a relatively
compact topological fundamental domain of the $\Gamma$-action on
$(X,\ell_0)$ such that its closure $\mc F =\clo{\mc F}_0$ is a
topological subgraph. (An example of such a topological fundamental
domain is given in \Fig{hex-graph}~(a) below.)  There is a canonical
spectral distribution function $N(\lambda)$, associated to the family
$H_\omega$, given by the trace formula
\begin{equation}
  \label{eq:idsn}
  N(\lambda)
  := \frac{1}{\Exp { \vol({\mathcal F},\ell_\bullet) }}
    \Exp{ \tr_\bullet [\1_{\mathcal F} 
                P_\bullet((-\infty,\lambda])] }, 
\end{equation}
where $\Exp {\cdot}$ denotes the expectation in $(\Omega,\Prob)$,
$\tr_\omega$ is the trace on the Hilbert space $\Lsqr
{X,\ell_\omega}$, and $P_\omega(I)$ denotes the spectral projection
associated to $H_\omega$ and the interval $I \subset \R$. Moreover,
the volume $\vol(\mc F,\ell_\bullet)$ is defined in~\eqref{eq:voldef}.
The function $N$ is called the \emph{(abstract) integrated density of
  states} with abbreviation IDS.

In the case of an amenable group $\Gamma$ the abstract IDS can also be
obtained via appropriate exhaustions. This is the statement of
\Thm{ids} below. A discrete group $\Gamma$ is called \emph{amenable},
if there exist a sequence $I_n \subset \Gamma$ of finite, non-empty
subsets with
\begin{equation}
 \label{eq:Foelner}
 \lim_{n\to\infty} \frac{|I_n\, \Delta\, I_n \gamma|}{|I_n|} = 0,  \quad 
 \text{ for all $\gamma \in \Gamma$.}
\end{equation}
A sequence $I_n$ satisfying~\eqref{eq:Foelner} is called a
\emph{F{\o}lner sequence}.

For every non-empty finite subset $I \subset \Gamma$, we define $
\Lambda(I) := \bigcup_{\gamma \in I} \gamma \mathcal F$. A sequence
$I_n \subset \Gamma$ of finite subsets is F\o lner if and only if the
associated sequence $\Lambda_n = \Lambda(I_n)$ of topological
subgraphs satisfies the \emph{van Hove condition}
\begin{equation}
  \label{eq:isop0}
  \lim_{n \to \infty} \frac{| \bd \Lambda(I_n)|}
                      {\vol(\Lambda(I_n),\ell_0)}
  = 0. 
\end{equation}
The proof of this fact is analogous to the proof of~\cite[Lemma
2.4]{peyerimhoff-veselic:02} in the Riemannian manifold case. Note
that~\eqref{eq:isop0} still holds if we replace $\bd \Lambda(I_n)$ by
$\bd_r \Lambda(I_n)$ for any $r \ge 1$, where $\bd_r \Lambda$ denotes
the \emph{thickened combinatorial boundary} $\set{ v \in V}{d(v,\bd
  \Lambda) \le r}$
and $d$ denotes the \emph{combinatorial distance}
which agrees (on the set of vertices) with the distance function of
the unilateral metric graph $(X,\ell_0)$.

A F\o lner sequence $I_n$ is called \emph{tempered}, if we additionally
have
\begin{equation} \label{eq:tempfol}
  \sup_{n \in \N} \frac{|\bigcup_{k \le n} I_{n+1} I_k^{-1}|}{|I_{n+1}|} 
< \infty. 
\end{equation}
Tempered F\o lner sequences are needed for an ergodic theorem of Lindenstrauss
\cite{lindenstrauss:01}. This ergodic theorem plays a crucial role in the
proof of \Thm{ids} presented below. However, the additional
property~\eqref{eq:tempfol} is not very restrictive since it was also shown in
\cite{lindenstrauss:01} that every F\o lner sequence $I_n$ has a tempered
subsequence $I_{n_j}$.

For any compact topological subgraph $\Lambda$ of $X$, we denote the
operator with Dirichlet vertex conditions on the boundary vertices
$\bd \Lambda$ and with the original vertex conditions $L_\omega(v)$ on
all inner vertices $v \in V(\Lambda) \setminus \bd \Lambda$ by
$H_\omega^{\Lambda,\Dir}$.  The label $\Dir$ refers to the Dirichlet
conditions on $\bd \Lambda$. For a precise definition of the Dirichlet
operator via quadratic forms, we refer to \Sec{approx}.  The spectral
projection corresponding to $H_\omega^{\Lambda,\Dir}$ is denoted by
$P_\omega^{\Lambda,\Dir}$. It is well-known that compactness of
$\Lambda$ implies that the operator $H_\omega^{\Lambda,\Dir}$ has
purely discrete spectrum.  The \emph{normalised eigenvalue counting
  function} associated to the operator $H_\omega^{\Lambda,\Dir}$ is
defined as
\begin{equation*}
  N_\omega^\Lambda(\lambda)
  = \frac{1}{\vol(\Lambda,\ell_\omega)} 
      \tr_\omega [P_\omega^{\Lambda,\Dir}((-\infty,\lambda])].
\end{equation*}
The function $N_\omega^\Lambda$ is the distribution function of a
(unique) pure point measure which we denote by $\mu_\omega^\Lambda$.

If $\Lambda = \Lambda(I_n)$ is associated to a F\o lner sequence $I_n
\subset \Gamma$, we use the abbreviations $H_\omega^{n,\Dir} :=
H_\omega^{\Lambda(I_n),\Dir}$ for the Schr\"odinger operator with
Dirichlet conditions on $\bd \Lambda(I_n)$,
$N_\omega^n:=N_\omega^{\Lambda(I_n)}$ for the normalised eigenvalue
counting function and $\mu_\omega^n:= \mu_\omega^{\Lambda(I_n)}$ for
the corresponding pure point measure on $\Lambda(I_n)$. We can now
state our first main result:
\begin{theorem}
  \label{thm:ids}
  Let $(X,\Omega,\Prob,\ell,L,q)$ be a random length covering model as
  described in \Ass{covgraph} with \emph{amenable} covering group
  $\Gamma$. Let $N$ be the IDS of the operator family $H_\omega$.
  Then there exist a subset $\Omega_0 \subset \Omega$ of full $\Prob$-measure
  such that we have, for every \emph{tempered} F\o lner sequence
  $I_n \subset \Gamma$,
  \begin{equation*}
    \lim_{n \to \infty} N_\omega^n(\lambda) = N(\lambda) 
  \end{equation*}
  for all $\omega \in \Omega_0$ and all points $\lambda \in \R$ at which $N$
  is continuous.
\end{theorem}
The proof is given in \Sec{approx}.
\begin{remark}
  The proof of \Thm{ids} yields even more. Let $\mu$ denote the
  measure associated to the distribution function $N$. Then we have
  \begin{equation}
    \label{eq:idsmeasconv} 
    \lim_{j \to \infty} \mu_\omega^n(f) = \mu(f)
  \end{equation}
  for all $\omega \in \Omega_0$ and all functions $f$ of the form
  $f(x) = g(x)(x+1)^{-1}$ with a function $g$ continuous on
  $[0,\infty)$ and with limit at infinity. (The behaviour of $g(x)$
  for $x < 0$ is of no importance since the spectral measures of all
  operators under consideration are supported on $\R^+ = [0,\infty)$.)
\end{remark}

\subsection{Wegner estimate}

In this subsection, we state a linear Wegner estimate for Laplace
operators of a random length model with independently distributed edge
lengths and fixed Kirchhoff vertex conditions. This Wegner estimate is
linear both in the number of edges and in the length of the considered
energy interval. {As mentioned in the introduction, a similar result
  for the case $\Z^d$ was proved recently by different methods in
  \cite{klopp-pankrashkin:09}}. In contrast to the previous
subsection, we do not require periodicity of the graph $X$ associated
to a group action.  More precisely, we assume the following:

\begin{assumption}
  \label{ass:wegner}
  Let $(X,\Omega,\Prob,\ell,L,q)$ be a random length model with the
  following properties:
  \begin{enumerate}
  \item We have $q \equiv 0$, i.e., the random family of operators are
    just the Laplacians ($H_\omega = \lapl {\omega}$) and we have no
    randomness in the vertex condition by fixing $L$ to be Kirchhoff
    in all vertices. Thus it suffices to look at the tuple
    $(X,\Omega,\Prob,\ell)$.

  \item We have a uniform upper bound $d_{\max} < \infty$ on the
    vertex degrees $\deg v$, $v \in V(X)$.
  
 \item Since the only randomness occurs in the edge lengths satisfying
    \begin{equation*}
      0 < \ell_{\min} \le \ell_\omega(e) \le \ell_{\max} 
         \quad \text{for all $\omega \in \Omega$ and $e \in E(X)$,}
    \end{equation*}
    we think of the probability space $\Omega$ as a Cartesian product
    $\prod_{e \in E} [\ell_{\min},\ell_{\max}]$ with projections
    $ \Omega \ni\omega \mapsto \omega_e = \ell_\omega(e) \in
    [\ell_{\min},\ell_{\max}]$.  The measure $\Prob$ is assumed to be a
    product $\bigotimes_{e \in E} \Prob_e$ of probability measures
    $\Prob_e$. Moreover, for every $e \in E$, we assume that $\Prob_e$
    is absolutely continuous with respect to the Lebesgue measure on
    $[\ell_{\min},\ell_{\max}]$ with density functions $h_e \in
    \Cont[1] \R$ satisfying
    \begin{equation}
      \label{eq:Cg.ex}
      \norm[\infty] {h_e}, \norm[\infty] {h_e'} \le C_h,
    \end{equation}
    for a constant $C_h > 0$ independent of $e \in E$.
  \end{enumerate}
\end{assumption}
Recall that $\tr_\omega$ is the trace in the Hilbert space $\Lsqr
{\Lambda,\ell_\omega}$. In the next theorem $P_\omega^{\Lambda,\Dir}$
denotes the spectral projection of the Laplacian
$\Delta_{\omega}^{\Lambda,\Dir}$ on $(\Lambda,\ell_\omega)$ with
Kirchhoff vertex conditions on all interior vertices and Dirichlet
boundary conditions on $\bd \Lambda$.  Under these assumptions we
have:
\begin{theorem}
  \label{thm:wegner}
  Let $(X,\Omega,\Prob,\ell)$ be a random length model satisfying
  \Ass{wegner}. Let $u > 1$ and $J_u = [1/u,u]$. Then there exists a
  constant $C>0$ such that
  \begin{equation*}
    \Exp {\tr P_\bullet^{\Lambda,\Dir}(I)}
     \le C \cdot \leb (I) \cdot \abs{E(\Lambda)}
  \end{equation*}
  for all compact subgraphs $\Lambda \subset X$ and all compact
  intervals $I \subset J_u$, where $\leb(I)$ denotes the
  Lebesgue-measure of $I$, and where $\abs{E(\Lambda)}$ denotes the
  number of edges in $\Lambda$.  The constant $C > 0$ depends only the
  constants $u$, $d_{\max}$, $\ell_{\min}$, $\ell_{\max}$ and the
  bound $C_h > 0$ associated to the densities $h_e$
  (see~\eqref{eq:Cg.ex}).
\end{theorem}
The proof will be given in \Sec{wegner}. We finish this section with
the following corollary. Recall that the periodic situation is a
special case of a random length covering model (see
\Rem{trivrandfam}):
\begin{corollary}
  \label{cor:wegner}
  Let $(X,\Omega,\Prob,\ell)$ be a random length covering model,
  satisfying both \Asss{covgraph}{wegner}, with \emph{amenable}
  covering group $\Gamma$. Then the IDS $N$ of the Laplacians
  $\Delta_\omega$ is a continuous function on $\R$ and even Lipschitz
  continuous on $(0,\infty)$.
\end{corollary}
\begin{proof}
  The Lipschitz continuity of $N$ on $(0,\infty)$ follows immediately
  from \Thms{ids}{wegner}. It remains to prove continuity of $N$ on
  $(-\infty,0]$. Note that our model is a special situation of the
  general \emph{ergodic groupoid} setting given in~\cite{LenzPV-07}.
  Thus, $N$ is the distribution function of a spectral measure of the
  direct integral operator $\int_\Omega^\oplus \lapl \omega\, \dd
  \Prob(\omega)$. Since $\lapl \omega \ge 0$ for all $\omega$,
  $N(\lambda)$ vanishes for all $\lambda < 0$. Moreover, if $N$ would
  have a jump at $\lambda = 0$, then $\ker \lapl \omega$ would be
  non-trivial for almost all $\omega \in \Omega$. But $\lapl \omega f
  = 0$  implies 
  \begin{equation*}
     0 = \iprod f {\Delta_\omega f}
       = \int_X \Bigabssqr{\frac{df}{dx}(x)} \dd x
  \end{equation*}
  since $\lapl \omega$ has Kirchhoff vertex conditions.  Thus $f$ is a
  constant function. Now $X$ is connected as well as non-compact,
  which implies that $\vol(X,\ell_\omega) = \infty$ by the lower bound
  $\ell_{\min}$ on the lengths of the edges.  Hence constant functions
  are not in $\Lsqrsymb$. This gives a contradiction.
\end{proof}

Our result on Lipschitz continuity of $N$ on $(0,\infty)$ is optimal
in the following sense:
\begin{remark}
  \label{rem:1d.lapl}
  It is well-known that the IDS of the free Laplacian $\lapl \R$ on
  $\R$ is proportional to the square root of the energy.  Note that
  this does not change when adding Kirchhoff boundary conditions at
  arbitrary points. Therefore, every model satisfying
  \Asss{covgraph}{wegner} for a metric graph isometric to $\R$ has in
  fact the above IDS.  Therefore, we cannot expect Lipschitz
  continuity of the IDS at zero for random length models without further
  assumptions.
\end{remark}

\section{Kagome lattice as an example of a planar graph}
\label{sec:example}

In this section, we illustrate the concepts of the previous section
for an explicit example. We introduce a particular regular
tessellation of the Euclidean plane admitting finitely supported
eigenfunctions of the combinatorial Laplacian. We discuss in detail
the discontinuities of the IDS of the combinatorial Laplacian and of
the Kirchhoff Laplacian of the induced equilateral metric graph. On
the other hand, applying \Cor{wegner}, we see that the IDS of a random
family of Kirchhoff Laplacians for independent distributed edge
lengths is continuous. Thus, randomness leads to an improvement of the
regularity of the IDS in this example.

We consider the infinite planar topological graph $X \subset \C$ as
illustrated in \Fig{plagraph}. This graph is sometimes called
\emph{Kagome lattice}. Every vertex of $X$ has degree four and belongs
to a uniquely determined upside triangle. Introducing $w_1 = 1$ and
$w_2 = \e^{\pi \im/3}$, we can identify the lower left vertex of a
particular upside triangle with the origin in $\C$ and its other two
vertices with $w_1, w_2 \in \C$. Consequently, the vertex set of $X$
is given explicitly as the disjoint union of the following three sets:
\begin{equation*}
   V(X) = (2 \Z w_1 + 2 \Z w_2)\, \dcup\, 
          (w_1 + 2 \Z w_1 + 2 \Z w_2)\, \dcup\, 
          (w_2 + 2 \Z w_1 + 2 \Z w_2).
\end{equation*}
A pair $v_1,v_2 \in V = V(X)$ of vertices is connected by a straight
edge if and only if $\abs{v_2 - v_1} =1$. We write $v_1 \sim v_2$ for
adjacent vertices. The above realisation of the planar graph $X
\subset \C$ is an isometric embedding of the metric graph
$(X,\ell_0)$.
 
\begin{figure}[h]
  \begin{center}
    \psfrag{w1}{$w_1$} 
    \psfrag{w2}{$w_2$}
    \includegraphics{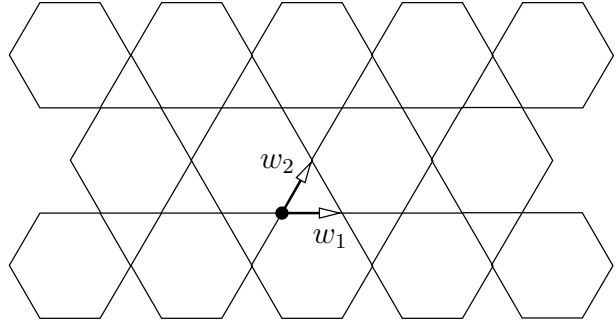}
  \end{center}
  \caption{Illustration of the planar graph $X$ (Kagome lattice).}
  \label{fig:plagraph}
\end{figure}

The group $\Z^2$ acts on $X$ via the maps $T_\gamma(x) := 2\gamma_1
w_1 + 2\gamma_2 w_2 + x$. A topological fundamental domain ${\mathcal
  F}_0$ of $X$ is thickened in \Fig{hex-graph}~(a). The set of
vertices of the topological subgraph ${\mathcal F} =
\overline{{\mathcal F}_0}$ (obtained by taking the closure of 
${\mathcal F}_0$ considered as a subset of the metric space $(X,\ell)$)
is given by $\{ a,b,c,a',b',b'',c'' \}$.

Note that we have to distinguish carefully between a topological and a
combinatorial fundamental domain. Let $G$ denote the underlying
combinatorial graph with set $V$ of vertices and $E$ of combinatorial
edges. The maps $T_\gamma$ act also on the set of vertices $V$ and a
combinatorial fundamental domain is given by $Q = \{ a,b,c \}$.  We
denote the translates $T_\gamma(Q)$ of $Q$ by $Q_\gamma$.

\begin{figure}[h]
  \begin{center}
    \psfrag{$2w_1$}{$2w_1$} 
    \psfrag{$2w_2$}{$2w_2$}
    \psfrag{$a$}{$a$}
    \psfrag{$b$}{$b$}
    \psfrag{$c$}{$c$}
    \psfrag{$a'$}{$a'$}
    \psfrag{$b'$}{$b'$}
    \psfrag{$b''$}{$b''$}
    \psfrag{$c''$}{$c''$}
    \psfrag{$b'''$}{$b'''$}
    \psfrag{(a)}{(a)}
    \psfrag{(b)}{(b)}
    \psfrag{$v_1$}{$v_1$}
    \psfrag{$v_2$}{$v_2$}
    \psfrag{$v_3$}{$v_3$}
    \psfrag{$v_4$}{$v_4$}
    \psfrag{$v_5$}{$v_5$}
    \psfrag{$w_1$}{$w_1$}
    \psfrag{$w_2$}{$w_2$}
    \psfrag{H0}{$H_{\gamma_0}$}
    \includegraphics{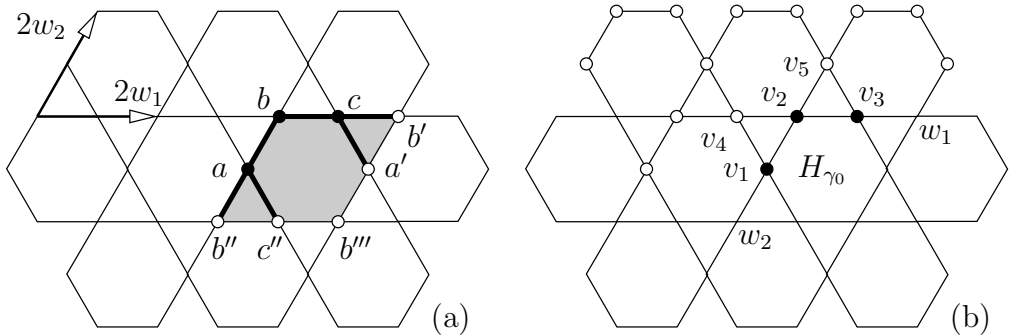}
  \end{center}
  \caption{(a) The periodic graph with thickened topological
    fundamental domain ${\mathcal F}_0$ and combinatorial fundamental
    domain $Q = \{ a,b,c \}$ (b) If $\gamma_0$ is vertically extremal
    for $F$, all white encircled vertices are zeroes of $F$.}
  \label{fig:hex-graph}
\end{figure}

\subsection{Spectrum and IDS of the combinatorial Laplacian}

We first observe that $G$ admits finitely supported eigenfunctions of
the \emph{combinatorial Laplacian} $\lapl{\comb}$: Choose an arbitrary
hexagon $H \subset X$ with vertices $\{ u_0,u_1,\dots,u_5 \}$. Then
there exists a centre $w_0 \in \C$ of $H$ such that we have
\begin{equation*}
  \{ u_0, u_1, \dots, u_5 \} 
  = \set{w_0 + \e^{k \pi \im/3}}{k = 0, 1, \dots, 5}.
\end{equation*}
The following function $\map{F_H} V {\{0,\pm1\}}$ on the vertices
\begin{equation} 
  \label{eq:FH}
  F_H(v) := 
  \begin{cases} 
     0, & \text{if $v \in V \setminus \{u_0,\dots,u_5\}$,}\\ 
    (-1)^k, & \text{if $v = w_0 + \e^{k \pi \im/3}$,}
  \end{cases}
\end{equation}
satisfies
\begin{equation*}
  \lapl \comb F_H(v) 
  = \frac 1 {\deg(v)} \sum_{w \sim v} (F_H(v) - F_H(w)) 
  = \frac32 F_H(v).
\end{equation*}
Thus, the vertices of every hexagon $H \subset X$ are the support of a
combinatorial eigenfunction $\map{F_H} V \R$. The functions $F_H$ are
the only finitely supported eigenfunctions up to linear combinations:

\begin{proposition}
  \label{prp:fineigen}
  \begin{enumerate}[(a)]
  \item
    \label{fin.a}
    Let $\map F V \R$ be a combinatorial eigenfunction on $X$ with
    finite support $\supp F \subset V$. Then
    \begin{equation*}
      \lapl \comb F = \frac32 F
    \end{equation*}
    and $F$ is a linear combination of finitely many eigenfunctions
    $F_H$ of the above type~\eqref{eq:FH}.
  \item
    \label{fin.b}
    Let $H_i \, (i=1, \dots, k)$ be a collection of distinct, albeit
    not necessarily disjoint, hexagons, and $F_i:=F_{H_i}$ the
    associated compactly supported eigenfunctions.  Then the set $F_1,
    \dots, F_k$ is linearly independent.
  \item
    \label{fin.c}
    If $ g \in \ell^2(V)$ satisfies $\Delta_{comb} g = \mu g$, then
    $\mu = 3/2$.
  \item
    \label{fin.d}
    The space of $\ell^2(V)$-eigenfunctions to the eigenvalue $3/2$ is
    spanned by compactly supported eigenfunctions.
  \end{enumerate}
\end{proposition}

\begin{proof} To prove~\eqref{fin.a}, assume that $\map F V \R$ is a
  finitely supported eigenfunction. Let $Q = \{ a,b,c \}$ be a
  combinatorial fundamental domain of $\Z^2$, as illustrated in
  \Fig{hex-graph}~(a) and $Q_\gamma := T_\gamma(Q)$. Let $H_\gamma$ be
  the uniquely defined hexagon containing the three vertices
  $Q_\gamma$. Moreover, we define
  \begin{equation*}
    A_0 := \set{\gamma \in \Z^2}{\supp F \cap Q_\gamma \neq \emptyset}.
  \end{equation*}
  Let $\eps_1 = (1,0)$ and $\eps_2 = (0,1)$. We say that $\gamma_0
  =(\gamma_{01},\gamma_{02}) \in A_0$ is \emph{vertically extremal} for $F$,
  if the second coordinate $\gamma_{02}$ is maximal amongst all
  $\gamma \in A_0$ and if $\gamma_0 - \eps_1 \notin A_0$. This means
  that $F$ vanishes in the left neighbour of $Q_{\gamma_0}$ and in all
  vertices vertically above $Q_{\gamma_0}$. Hence, $\gamma_0$ in
  \Fig{hex-graph} (b) is vertically extremal if $F$ vanishes in all
  white encircled vertices and does not vanish in at least one of the black
  vertices. Obviously, $A_0$ has always vertically extremal elements.
  Choosing such a $\gamma_0 \in A_0$, we will show below that $F$ is an
  eigenfunction with eigenvalue $3/2$ and that the following
  facts hold:
  \begin{itemize}
  \item[(i)] $\gamma_0 + \eps_1$ belongs to $A_0$,
  \item[(ii)] $\gamma_0 - \eps_2$ or $\gamma_0 - \eps_2 - \eps_1$
    belong to $A_0$,
  \item[(iii)] adding a suitable multiple of $F_{H_{\gamma_0}}$ to
    $F$, we obtain a new eigenfunction $F_1$ and a set $A_1 :=
    \set{\gamma \in \Z^2}{\supp F_1 \cap Q_\gamma \neq \emptyset}$
    satisfying
    \begin{equation*}
      \gamma_0 \notin A_1, \quad 
      A_1 \setminus A_0 \subset 
         \{ \gamma_0-\eps_2, \gamma_0+\eps_1-\eps_2 \}.
    \end{equation*}
  \end{itemize}

  To see this, let $\gamma_0 \in A_0$ be vertically extremal and
  $v_1,\dots,v_5,w_1,w_2$ be chosen as in \Fig{hex-graph}~(b). The
  eigenvalue equation at the vertices $v_4$ and $v_5$, in which $F$
  vanishes, imply that we have $F(v_1) = - F(v_2) = F(v_3) \neq 0$.
  Applying the eigenvalue equation again, now at $v_2$, yields that
  the eigenvalue of $F$ must be $3/2$.

  If $\gamma_0 + \eps_1 \notin A_0$, $F$ would vanish in $w_1$ and
  all its neighbours, except for $v_3$. This would contradict to the
  eigenvalue equation at $w_1$ and~(i) is proven. Similarly, if
  $\gamma_0 -\eps_2, \gamma_0 - \eps_2 - \eps_1 \notin A_0$, we would
  obtain a contradiction to the eigenvalue equation at the vertex
  $w_2$. This proves~(ii).
  
  By adding $F(v_1) F_{H_{\gamma_0}}$ to $F$, we obtain a new
  eigenfunction $F_1$ (again to the eigenvalue $3/2$) which vanishes
  at all vertices of $Q_{\gamma_0} = \{ v_1, v_2, v_3 \}$. Thus we
  have $\gamma_0 \notin A_1$. But $F$ and $F_1$ differ only in the
  vertices $Q_{\gamma_0}$, $Q_{\gamma_0+\eps_1}$,
  $Q_{\gamma_0-\eps_2}$ and $Q_{\gamma_0+\eps_1-\eps_2}$, establishing
  property~(iii).

  The above procedure can be iteratively (from left to right) applied
  to the hexagons in the top row of $A_0$: Step (iii) can be applied
  to the function $F_1$ and a vertically extremal element of $A_1$.
  After a finite number $n$ of steps the top row of hexagons in $A_0$
  is no longer in the support of the function $F_n$. (Note that
  property (i) implies that when removing the penultimate hexagon form
  the right, one has simultaneously removed the rightermost one, too.)
  Again, this procedure can be iterated removing successively rows of
  hexagons. This time property (ii) guarantees that the procedure
  stops after a finite number $N$ of steps with $F_N\equiv 0$. We have
  proven statement~\eqref{fin.a}.

  Now we turn to the proof of~\eqref{fin.b}.  Since the graph is
  connected there exists a vertex $v$ in $A:=\cup_{i=1}^k H_i$ which
  is adjacent to some vertex outside $A$. Then $v$ is contained in
  precisely one hexagon $H_{i_0}$.  (In the full graph each vertex is
  in two hexagons.)  Thus the condition
  \begin{equation}
    \label{e_lin-comb} 
    \sum_{i=1}^k \alpha_i F_i = 0   \quad \alpha_i \in \CC
  \end{equation}
  evaluated at the vertex $v$ implies $\alpha_{i_0}=0$. This shows
  that all coefficients $\alpha_i$ in \eqref{e_lin-comb} corresponding
  to hexagons $H_i$ lying at the boundary of $A$ vanish. This leads to
  an equation analogous to \eqref{e_lin-comb} where the indices in the
  sum run over a strict subset of $\{1,\dots,k\}$.  Now one iterates
  the pocedure and shows that actually all coefficients
  $\alpha_1,\dots,\alpha_k$ in \eqref{e_lin-comb} are zero. We have
  shown linear independence of $F_1, \dots, F_k$.

  To prove~\eqref{fin.c} we recall that the IDS $\Delta_{\rm comb}$ is
  a spectral measure (see e.g.~\cite[Prop.~5.2]{LenzPV-07}). Thus the
  IDS jumps at the value $\mu$.  This in turn implies by
  \cite[Prop.~5.2]{Veselic-05b} that there is a compactly supported
  $\tilde g$ satisfying the eigenvalue equation. Now~\eqref{fin.a}
  implies $\mu = 3/2$.

  Statement~\eqref{fin.d} follows from~\cite[Thm.~2.2]{LenzV}, cf.
  also the proof of \Prp{combids}.
\end{proof}

We are primarily interested in $\lsqrsymb$-eigenfunctions of $\lapl
\comb$, since their eigenvalues coincide with the
\emph{discontinuities} of the corresponding IDS. For combinatorial
covering graphs with \emph{amenable} covering group $\Gamma$, every
$\lsqrsymb$-eigenfunction $F$ implies the existence of a
\emph{finitely supported} eigenfunction to the same eigenvalue which
is implied, e.g., by~\cite[Prop.~5.2]{Veselic-05b} or
\cite[Thm.~2.2]{LenzV}.  (Related, but different results have been
obtained before in~\cite{mathai-yates:02}. If the group is even 
abelian, as is the case for the Kagome lattice, the analogous result 
was proven even earlier in~\cite{kuchment:91}.) 
 It should be mentioned
here that the situation is very different in the smooth category of
Riemannian manifolds. There, compactly supported eigenfunctions cannot
occur due to the \emph{unique continuation principle}. In the discrete
setting of graphs, non-existence of finitely supported combinatorial
eigenfunctions is --- at present --- only be proved for particular
examples or in the case of planar graphs of non-positive combinatorial
curvature; see~\cite{klps:06} for more details.  Hence, \Prp{fineigen}
tells us that $X$ does not admit combinatorial
$\lsqrsymb$-eigenfunctions associated to eigenvalues $\mu \neq 3/2$.

Next, let us discuss spectral informations which can be obtained with
the help of \emph{Floquet theory}. Using a general result of Kuchment
(see~\cite{kuchment:91} or~\cite[Thm.~8]{kuchment:05}) for periodic
finite difference operators (applying Floquet theory to such
operators) we conclude that the \emph{compactly supported}
eigenfunctions of $\lapl \comb$ associated to the eigenvalue $3/2$ are
already dense in the whole eigenspace $\ker (\lapl \comb - 3/2)$.
As for the whole spectrum, we derive the following result:

\begin{proposition}
  \label{prp:combspec} 
  Denote by $\spec[ac]{\lapl \comb}$ and $\spec[p]{\lapl \comp}$ the
  absolutely continuous and point spectrum of $\lapl \comb$ on our
  $\Z^2$-periodic graph $X$. Then we have
  \begin{equation*}
    \spec[ac]{\lapl \comb}
    = \Bigl[0,\frac{3}{2} \Bigr] \quad  \text{and} \quad
    \spec[p]{\lapl \comb} 
    = \Bigl\{ \frac 32 \Bigr\}.
  \end{equation*}
\end{proposition}

The proof follows from standard Floquet theory (for a similar
hexagonal graph model see~\cite{kuchment-post:07}):

\begin{proof}
  Note that we have the unitary equivalence
  \begin{equation*}
    \lapl \comb \cong \int_{\Torus^2}^\oplus \lapl[\theta] \comb \dd \theta,
  \end{equation*}
  where $\lapl[\theta] \comb$ is the $\theta$-equivariant Laplacian on
  $Q$, $\theta \in \Torus^2:=\R^2/(2\pi \Z)^2$. This operator is
  equivalent to the matrix
  \begin{equation*}
    \lapl[\theta] \comb \cong \frac 14
    \begin{pmatrix}
      4 & -1-\e^{-\im \theta_2} & -\e^{-\im \theta_1}-\e^{-\im \theta_2}\\
      -1-\e^{\im \theta_2} & 4 & -1-\e^{-\im \theta_1}\\
      -\e^{\im \theta_1}-\e^{\im \theta_2} & -1-\e^{\im \theta_1} & 4
    \end{pmatrix}
  \end{equation*}
  using the basis $F \cong (F(a),F(b),F(c))$ for a function on $Q$ and
  the fact that $F(T_\gamma v)= \e^{\im \iprod \theta \gamma} F(v)$
  (equivariance). The characteristic polynomial is
  \begin{equation*}
    p(\mu) = \Bigl(\mu - \frac 32 \Bigr) 
           \bigg( \Bigl(\mu - \frac 34 \Bigr)^2 
           - \frac{3+2\kappa}{16}\bigg),
  \end{equation*}
  where $\kappa=\cos\theta_1 + \cos \theta_2 + \cos
  (\theta_1-\theta_2)$, and the eigenvalues of $\lapl[\theta] \comb$
  are
  \begin{equation*}
    \mu_1 = \frac 32 \quad \text{and }
    \mu_\pm = \frac 34 \pm \frac 14\sqrt{3 + 2\kappa}.
  \end{equation*}
  In particular, we recover the fact that $\lapl \comb$ has an
  eigenfunction, since $\mu_1$ is independent of $\theta$, only
  $\mu_\pm$ depend on $\theta$ via $\kappa=\kappa(\theta)$. Note
  that we have
  \begin{equation*}
    - \frac 32
    = \kappa \Bigl( \frac{2\pi} 3, \frac{4\pi} 3 \Bigr)
    \le \kappa(\theta)
    \le \kappa(0,0)
    = 3,
  \end{equation*}
  giving the spectral bands $B_- = [0,3/4]$ and $B_+ = [3/4, 3/2]$.
\end{proof}

The next result discusses (dis)continuity properties of the IDS
associated to the combinatorial Laplacian on $X$: 

\begin{proposition}
  \label{prp:combids}
  Let $N_\comb$ be the (abstract) IDS of the $\Z^2$-periodic operator
  $\lapl \comb$, given by
  \begin{equation*}
    N_\comb(\mu) = \frac{1}{|Q|} \tr [\1_Q P_\comb((-\infty,\mu])],
  \end{equation*}
  where $\tr$ is the trace on the Hilbert space $\lsqr {V}$ and
  $P_\comb$ denotes the spectral projection of $\lapl \comb$. Then
  $N_\comb$ vanishes on $(-\infty,0]$, is continuous on $\R \setminus
  \{ 3/2 \}$ and has a jump of size $1/3$ at $\mu = 3/2$. Moreover,
  $N_\comb$ is strictly monotone increasing on $[0,3/2]$ and
  $N_\comb(\mu) = 1$ for $\mu \ge 3/2$.
\end{proposition}

\begin{proof}
  The following facts are given, e.g.,
  in~\cite[p.~119]{mathai-yates:02}:
  \begin{enumerate}
  \item the points of increase of $N_\comb$ coincide with the spectrum
    $\sigma(\lapl \comb)$ and
  \item $N_\comb$ can only have discontinuities at $\sigma_p(\lapl
    \comb)$.
  \end{enumerate}
  Together with \Prp{combspec}, all statements of the proposition
  follow, except for the size of the jump at $\mu = 3/2$.

  Let us choose a F\o lner sequence $I_n \subset \Z^2$ and define
  $\Lambda_n = \bigcup_{\gamma \in I_n} Q_\gamma$.  Let $\bd
  \Lambda_n$ denote the set of boundary vertices of the combinatorial
  graph induced by the vertex set $\Lambda_n$, and
   \begin{equation}
     \label{eq:comb.bd}
     \bd_r \Lambda_n := \set{v \in V(X)}{d(v,\bd \Lambda_n) \le r}
   \end{equation}
   be the thickened (combinatorial) boundary.  Let
  \begin{equation}
    \label{eq:dlambda}
    D(\mu) :=
    N_\comb(\mu) - \lim_{\eps \to 0}
       N_\comb(\mu -\eps)
    = \frac 1 {|\Lambda_n|} 
            \tr \bigl[\1_{\Lambda_n} P_\comb(\{\mu\})\bigr].
  \end{equation} 
  The last equality in~\eqref{eq:dlambda} holds for all $n$ and
  follows easily from the $\Z^2$-invariance of the operator $\lapl
  \comb$. It remains to prove that $D(3/2) = 1/3$. Let $\Lambda_n' =
  \Lambda_n \setminus \bd_1 \Lambda_n$ and
  \begin{equation*}
    D_n(\mu) := \frac{1}{| \Lambda_n |} \dim E_n(\mu),
  \end{equation*}
  where $E_n(\mu) := \set{F \in \ker(\lapl \comb - \mu)} {\supp F
    \subset \Lambda_n'}$. Arguments as
  in~\cite{msy:03} or in~\cite{LenzV} show
  that
  \begin{equation}
    \label{eq:DnD} 
    D(\mu) = \lim_{n \to \infty} D_n(\mu). 
  \end{equation}
  For the convenience of the reader, we outline the proof
  of~\eqref{eq:DnD} below. Using part (b) of Proposition
  \ref{prp:fineigen} one can show that $\dim E_n(\mu)$ equals up to a
  boundary term the number of hexagons contained in $\Lambda_n'$.
  Since every translated combinatorial fundamental domain $Q_\gamma$
  uniquely determines a hexagon $H_\gamma$ and $|Q| = 3$, we conclude
  that $\dim E_n(\mu) \approx \frac 13 | \Lambda_n |$, up to an error
  proportional to $| \bd_1 \Lambda_n |$. The van Hove
  property~\eqref{eq:isop0} (which holds also in the combinatorial
  setting) then implies the desired result $D(3/2) = \lim_{n \to
    \infty} D_n(3/2) = 1/3$.

  Finally, we outline the proof of~\eqref{eq:DnD}: Let $E(\mu) =
  \ker(\lapl \comb - \mu)$ and $S_n(\mu) = \1_{\Lambda_n} E(\mu)$. Let
  $\map {b_n} {S_n(\mu)} {\R^{|\bd_1 \Lambda_n|}}$ be the boundary
  map, i.e., $b_n(F)$ is the collection of all values of $F$ assumed
  at the (thickened) boundary vertices $\bd_1 \Lambda_n$. Then $\ker
  b_n = E_n(\mu) \subset S_n(\mu)$, and we have
  \begin{equation*}
    D_n(\mu) \le D(\mu)
    \le \frac{\dim S_n(\mu)}{|\Lambda_n|}
    = \frac{\dim \ker b_n}{|\Lambda_n|} 
    + \frac{\dim \ran\, b_n}{|\Lambda_n|}
    \le D_n(\mu) + \frac{|\bd_1 \Lambda_n|}{|\Lambda_n|},
  \end{equation*}
  which yields~\eqref{eq:DnD}, by taking the limit, as $n \to \infty$.
\end{proof}

\subsection{Spectrum and IDS of the periodic Kirchhoff Laplacian}

There is a well known correspondence between the spectrum $\spec{\lapl
  \comb}$ on a graph $G$ and the spectrum of the (Kirchhoff) Laplacian
$\lapl 0$ on the corresponding (equilateral) metric graph $(X,\ell_0)$
with $\ell_0 = \1_E$ (see e.g.~\cite{von-below:85,nicaise:85,
  cattaneo:97,bgp:08,post:08a} and the references therein).  Namely,
any $\lambda \neq k^2 \pi^2$ lies in $\spec[p] {\lapl 0}$ resp.\
$\spec[ac]{\lapl 0}$ iff $\mu(\lambda) = 1 - \cos \sqrt{\lambda}$ lies
in $\spec[p]{\lapl \comb}$ resp.\ $\spec[ac]{\lapl \comb}$. Moreover,
the eigenspace of the metric Laplacian is isomorphic to the
corresponding eigenspace of the combinatorial Laplacian.

Let $\map F V \C$ be a finitely supported eigenfunction of $\lapl
\comb$ as in the previous section. In particular, the eigenvalue must
be $\mu=3/2$.  The above mentioned correspondence shows that, for
every $\lambda = (2 k + 2/3)^2 \pi^2$, $k \in \Z$, (i.e.
$\mu(\lambda)=3/2)$), there is a Kirchhoff eigenfunction $\map f X \R$
of compact support associated to the eigenvalue $\lambda$, satisfying
$f(v) = F(v)$ at all vertices $v \in V$. In addition, if $\lambda =
k^2 \pi^2$, there are so-called \emph{Dirichlet eigenfunctions} of
$\lapl 0$, determined by the topology of the graph (see
e.g.~\cite{von-below:85,nicaise:85,kuchment:05, lledo-post:08b}),
which are also generated by compactly supported eigenfunctions.

Using the results~\cite{cattaneo:97,bgp:08}, we conclude from
\Prp{combspec}:
\begin{corollary}
  \label{cor:kirchspec}
  Let $\lapl 0$ denote the Kirchhoff Laplacian of the equilateral
  metric graph $(X,\ell_0)$. Let $\sigma_{\mathrm p}$ and
  $\sigma_{\mathrm{ac}}$ denote the point spectrum and absolutely
  continuous spectrum and $\sigma_\comp$ denote the spectrum given by
  the compactly supported eigenfunctions. Then we have
  \begin{equation*}
    \spec[comp]{\lapl 0}
    = \spec[p] {\lapl 0}
    = \Bigset{\Bigl(2k+\frac23 \Bigr)^2 \pi^2} {k \in \Z} 
    \cup \bigset{k^2 \pi^2} {k \in \N}
  \end{equation*}
  and
  \begin{equation}
    \label{eq:kirch.ac}
    \spec[ac] {\lapl 0}
       =\Bigl[ 0, \Bigl(\frac 23 \Bigr)^2 \pi^2 \Bigr]  \cup
        \bigcup_{k \in \N} 
            \Bigl[ \Bigl(2k-\frac 23 \Bigr)^2\pi^2,
                   \Bigl(2k+\frac 23 \Bigr)^2\pi^2 \Bigr].
  \end{equation}
\end{corollary}
Similarly, as in the discrete setting, we conclude the following
(dis)continuity properties of the IDS: 
\begin{proposition}
  \label{prp:kirchids} 
  Let $N_0$ be the (abstract) IDS of the $\Z^2$-periodic Kirchhoff
  Laplacian $\lapl 0$ on the metric graph $(X,\ell_0)$, given by
  \begin{equation*}
    N_0(\lambda)
    = \frac 1 {\vol({\mathcal F},\ell_0)} 
      \tr [\1_{\mc F} P_0((-\infty,\lambda]) ],
  \end{equation*}
  where $\tr$ is the trace on the Hilbert space $\Lsqr {X,\ell_0}$
  and $P_0$ denotes the spectral projection of $\lapl 0$. Then all the
  discontinuities of $\map {N_0} \R {[0,\infty)}$ are
  \begin{enumerate}
  \item at $\lambda = (2k + \frac 23)^2\pi^2$, $k \in \Z$, with jumps
    of size $\frac 16$,
  \item at $\lambda = k^2 \pi^2$, $k \in \N$, with jumps of size
    $\frac 12$.
  \end{enumerate}
  Moreover, $N_0$ is strictly monotone increasing on the absolutely
  continuous spectrum $\spec[ac]{\lapl 0}$ given
  in~\eqref{eq:kirch.ac} and $N_0$ is constant on the complement of
  $\spec {\lapl 0}$.
\end{proposition}

\begin{proof}
  Our periodic situation fits into the general setting given
  in~\cite{LenzPV-07}, by choosing the trivial probability space
  $\Omega = \{ \omega \}$ with only one element.  Proposition~5.2
  in~\cite{LenzPV-07} states that $N_0$ is the distribution function
  of a spectral measure for the operator $\lapl 0$. Consequently,
  discontinuities of $N_0$ can only occur at the
  $\Lsqrsymb$-eigenvalues of $\lapl 0$, and the points of increase of
  $N_0$ coincide with the spectrum $\sigma(\lapl 0)$, which is given
  in \Cor{kirchspec}. Hence, it only remains to prove the statements
  about the discontinuities of $N_0$. We know from~\cite[Theorem
  11]{kuchment:05} that the compactly supported eigenfunctions densely
  exhaust every $\Lsqrsymb$-eigenspace of $\lapl 0$.

  Let $I_n \subset \Z^2$ be a F\o lner sequence. This time, we look at
  the corresponding topological graphs $\Lambda(I_n)$ and their
  thickened topological boundaries $\bd_r \Lambda(I_n) = \set{x \in
    X}{d(x,\bd \Lambda(I_n)) \le r}$, and denote them by $\Lambda_n$
  and $\bd_r \Lambda_n$, respectively. We are interested in the jumps
  \begin{equation*}
    D(\lambda)
    :=  N_0(\lambda) -  \lim_{\eps \to 0} N_0(\lambda-\eps)
    = \frac 1 {\vol(\Lambda_n,\ell_0)}
       \tr \bigr[ \1_{\Lambda_n} P_0(\{\lambda\}) \bigl],
  \end{equation*}
  where the right hand side is, again, independent of the choice of
  $n$.  Let $\Lambda_n'$ be the closure of $\Lambda_n \setminus \bd_1
  \Lambda_n$ and
  \begin{equation*}
    D_n(\lambda)
    := \frac 1 {\vol(\Lambda_n,\ell_0)} \dim E_n(\lambda),
  \end{equation*}
  with $E_n(\lambda) = \set{f \in \ker(\lapl 0 - \lambda)} {\supp f
    \subset \Lambda_n'}$. Arguments analogously to the proof
  of~\eqref{eq:DnD} yield
  \begin{equation}
    \label{eq:DnDkirch} 
    D(\lambda) = \lim_{n \to \infty} D_n(\lambda). 
  \end{equation}
  For the proof of~\eqref{eq:DnDkirch}, however, we have to define the
  boundary map
  \begin{equation*}
    \map {b_n} {S_n(\lambda)} 
         {\bigoplus_{v \in \bd \Lambda_n} (\C \oplus \C^{E_v})}
    \quad \text{by} \quad
    (b_n f)_v:= (f(v), \orul
    {D f}(v)).
  \end{equation*}

  Let $\lambda = (2k+2/3)^2 \pi^2$, $k \in \Z$. We follow the same
  arguments as in the proof of \Prp{combids}.  Again, $\dim
  E_n(\lambda)$ is equal to the number of hexagons contained in
  $\Lambda_n$ up to a boundary term and
  we have $\vol(\mc F,\ell_0) = 6$ (see \Fig{hex-graph}~(a)).
  Therefore, we derive that the corresponding jump is of size $1/6$.

  Let $\lambda = k^2 \pi^2$, $k \in \N$. We know
  from~\cite{von-below:85,nicaise:85} or from~\cite[Lem.~5.1 and
  Prop.~5.2]{lledo-post:08b} that the dimension of $E_n(\lambda)$ is
  (up to an error proportional to $|\bd \Lambda_n|$) approximately
  equal to
  \begin{equation*}
    |E(\Lambda_n)| - |V(\Lambda_n)|
    \approx \frac 12
    \vol(\Lambda_n,\ell_0). 
  \end{equation*}
  This implies that $N_0$ has a discontinuity at $\lambda = k^2 \pi^2$
  of size $1/2$.
\end{proof}

\begin{remark}
  \label{rem:ids.gen.per.gr}
  Note that \Prps{combids}{kirchids} hold also for general covering
  graphs $X \to X_0$ with amenable covering group $\Gamma$ and compact
  quotient $X_0 \cong X/\Gamma$, once we have information about the
  shape of the support of elementary eigenfunctions (i.e.,
  eigenfunctions, which generate the eigenspace by linear combinations
  and translations). In our Kagome lattice example the elementary
  eigenfunction is supported on a hexagon. For example, the jump of
  size $1/3$ at the eigenvalue $\mu=3/2$ in the discrete case is the
  number $\nu$ of hexagons determined by a combinatorial fundamental
  domain ($\nu=1$) divided by the number of vertices in a
  combinatorial fundamental domain ($\abs Q=3$).

  In the metric graph setting, the jump at $\lambda=(2k+2/3)^2\pi^2$
  is of size $1/6$ due to the fact that we have six edges in one
  topological fundamental domain.

  For the eigenvalues at $\lambda=k^2\pi^2$ (also called
  \emph{topological}, see~\cite{lledo-post:08b}) we even have a
  precise information for any $r$-regular amenable covering graph,
  namely
  \begin{equation*}
    \dim E_n(\lambda)
    \approx |E(\Lambda_n)| - |V(\Lambda_n)|
    \approx \Bigl(1-\frac 2 r \Bigr) \abs{E(\Lambda_n)}
    =\Bigl(1-\frac 2 r \Bigr) \vol(\Lambda_n,\ell_0), 
  \end{equation*}
  up to an error proportional to $\abs{\bd \Lambda_n}$, so that the
  jump of $N_0$ at $\lambda$ is $(1-2/r)$.
\end{remark}

\subsection{IDS of associated random length models}

Finally, we impose a random length structure $\map \ell {\Omega \times
  E} {[\ell_{\min},\ell_{\max}]}$ on the edges of $(X,\ell_0)$ with
independently distributed edge lengths, as described in \Ass{wegner}.
Then \Cor{wegner} tells us that the associated integrated density of
states $\map N \R {[0,\infty)}$ is continuous and even Lipschitz
continuous on $(0,\infty)$. Hence, all discontinuities occurring for
the IDS of the Kirchhoff Laplacian on the $\Z^2$-periodic graph
$(X,\ell_0)$ disappear by introducing this type of randomness.

\section{Proof of the approximation of the IDS via exhaustions}
\label{sec:approx}

In this section, we prove \Thm{ids}, namely, that the non-random
integrated density of states~\eqref{eq:idsn} can be approximated by
suitably chosen normalised eigenvalue counting functions, for
$\Prob$-almost all random parameters $\omega \in \Omega$.

For the following considerations, we need the quadratic forms
associated to the Schr\"odinger operators. Recall that for each
Lagrangian subspace $L_v \subset \C^{E_v} \oplus \C^{E_v}$ describing
the vertex condition at $v \in V$ there exists a unique orthogonal
projection $Q_v$ on $\C^{E_v}$ with range $\mc G_v := \ran Q_v$ and a
symmetric operator on $\mc G_v$ such that~\eqref{eq:qg.vx.cond} holds.

Let $\Lambda \subset X$ be a topological subgraph. The quadratic form
associated to the operator with vertex conditions given by $(\mc G_v,
R_v)$ at inner vertices $V(\Lambda) \setminus \bd \Lambda$ and
Dirichlet conditions at $\bd \Lambda$ is defined as
\begin{gather*}
  \dom \qf h^{\Lambda,\Dir}
  = \bigset{f \in \Sobx \max {X,\ell}}
    {\ul f(v) \in \mc G_v \; \forall v \in V(\Lambda) \setminus \bd \Lambda,
      \; \ul f(v) = 0 \; \forall v \in \bd \Lambda},\\
  \qf h^{\Lambda,\Dir}(f)
  = \normsqr[\Lsqr{\Lambda,\ell}] {Df} 
     + \iprod[\Lsqr{\Lambda,\ell}] {q f} f
     + \sum_{v \in V(\Lambda)} \iprod[\mc G_v]{R_v \ul f(v)}{\ul f(v)}.
\end{gather*}
In particular, if $\Lambda=X$ is the full graph, then there is no
boundary and $\qf h=\qf h^X$ is the quadratic form associated to the
operator $H=H_{(X,\ell),L}$.

If $\ell_{\min}:=\inf_e \ell(e)>0$, $C_\pot:=\norm[\infty] q < \infty$
and $\sup_v \norm{R_v}=: C_R < \infty$, then $\qf h^{\Lambda,\Dir}$ is
a closed quadratic form with corresponding self-adjoint operator
$H^{\Lambda,\Dir}$.
\begin{lemma}
  \label{lem:lower.bd}
  For any subgraph $\Lambda$ of $X$, the quadratic form $\qf
  h^{\Lambda,\Dir}$ is closed. Moreover, the associated self-adjoint
  operator $H^{\Lambda,\Dir}$ has domain given by
  \begin{multline*}
    \dom H^{\Lambda,\Dir}
    = \Bigset{f \in \Sobx[2] \max {X,\ell}}
        {\ul f(v) = 0 \; \forall v \in \bd V,\\
          \ul f(v) \in \mc G_v,\; Q_v \orul{Df}(v)=R_v \ul f(v) \;
          \forall v \in V(\Lambda) \setminus \bd \Lambda}.
  \end{multline*}
  Moreover, $H^{\Lambda,\Dir}$ is \emph{uniformly} bounded from below
  by $-C_0$ where $C_0 \ge 0$ depends only on $\ell_-$, $C_R$ and
  $C_\pot$, but not on $\Lambda$.
\end{lemma}
\begin{proof}
  The first assertion follows from~\cite[Thm.~17]{kuchment:04}. The
  \emph{uniform} lower bound is a consequence
  of~\cite[Cor.~10]{kuchment:04} where the lower bound is given
  explicitly.  Basically, the statements follow from a standard
  Sobolev estimate of the type
  \begin{equation*}
    \Bigabs{\sum_v \iprod{R_v \ul f(v)}{\ul f(v)}}
    \le C_R \sum_{v \in V(\Lambda)} \abssqr {\ul f(v)}
    \le \eta \normsqr{Df} + C_\eta \normsqr f
  \end{equation*}
  for $\eta>0$, where $C_\eta$ depends only on $\eta$, $C_R$ and
  $\ell_{\min}$.
\end{proof}

The Dirichlet operator will serve as upper bound in the bracketing
inequality~\eqref{eq:brack.op} later on. In order to have a lower
bound we introduce a Neumann-type operator $H^\Lambda$ via its
quadratic form $\qf h^\Lambda$. Since the vertex conditions can be
negative, we have to use the boundary condition $(\C^{E_v},-C_R)$
instead of a simple Neumann boundary condition $(\C^{E_v},0)$.  The
quadratic form $\qf h^\Lambda$ is defined by
\begin{gather*}
  \dom \qf h^\Lambda
  = \bigset{f \in \Sobx \max {X,\ell}}
    {\ul f(v) \in \mc G_v \; 
                 \forall v \in V(\Lambda) \setminus \bd \Lambda},\\
  \qf h^\Lambda(f)
  = \normsqr[\Lsqr{\Lambda,\ell}] {Df} 
     + \iprod[\Lsqr{\Lambda,\ell}] {q f} f
     + \!\!\! \sum_{v \in V(\Lambda) \setminus \bd \Lambda} \!\!\!
               \iprod[\mc G_v]{R_v \ul f(v)}{\ul f(v)}
     - C_R \sum_{v \in \bd \Lambda} \abssqr[\mc G_v]{\ul f(v)}.
\end{gather*}
Note that the boundary condition $\wt R_v=-C_R$ trivially fulfills the
norm bound $\norm{\wt R_v} \le C_R$, and therefore by \Lem{lower.bd},
the form $\qf h^\Lambda$ is uniformly bounded from below by the same
constant $-C_0$ as $\qf h^{\Lambda,\Dir}$. By adding $C_0$ to the
(edge) potential $q$ we may assume that w.l.o.g.~$H^X$,
$H^{\Lambda,\Dir}$ and $H^\Lambda$ are all non-negative for all
subgraphs $\Lambda$.

We can now show the following bracketing result:
\begin{lemma}
  \label{lem:brack.op}
  Let $\Lambda$ be a topological subgraph of $X$ and $\Lambda'$ be the
  closure of the complement $\compl \Lambda$. Then
  \begin{equation}
    \label{eq:brack.op}
    H^{\Lambda,\Dir} \oplus H^{\Lambda',\Dir}
    \ge H
    \ge H^\Lambda \oplus H^{\Lambda'}
    \ge 0
  \end{equation}
  in the sense of quadratic forms.
\end{lemma}

\begin{proof}
  It is clear from the inclusions $\{ 0 \} \subset \mc G_v \subset
  \C^{E_v}$ for all boundary vertices $v \in \bd \Lambda$ that the
  quadratic form domains fulfil
  \begin{equation*}
    \dom \qf h^{\Lambda,\Dir} \oplus \dom \qf h^{\Lambda',\Dir}
    \subset
    \dom \qf h
    \subset
    \dom \qf h^\Lambda \oplus \dom \qf h^{\Lambda'}.
  \end{equation*}
  Moreover, if $f=f_\Lambda \oplus f_{\Lambda'}$ is in the decoupled
  Dirichlet domain, then
  \begin{equation*}
    \qf h^{\Lambda,\Dir}(f_\Lambda)
     + \qf h^{\Lambda',\Dir}(f_{\Lambda'})
    = \qf h(f)
  \end{equation*}
  since $\ul f(v)=0$ on boundary vertices, if $f \in \dom \qf h$, then
  \begin{equation*}
    \qf h(f) \ge 
    \qf h^\Lambda(f_\Lambda)
    + \qf h^{\Lambda'}(f_{\Lambda'})
  \end{equation*}
  since $R_v \ge -C_R$. In particular, we have shown the inequality
  for the quadratic forms.
\end{proof}

Next, we provide a useful lemma about the spectral shift function of
two operators. For a non-negative operator $H$ with purely discrete
spectrum $\set{\lambda_k(H)}{k \ge 0}$ (repeated according to
multiplicity), the eigenvalue counting function is given by
\begin{equation*}
  n(H,\lambda) :=
  \tr \1_{[0,\lambda)}(H) =
  \bigabs{\set{k \ge 0} {\lambda_k(H) \le \lambda}}.
\end{equation*}
The \emph{spectral shift function (SSF)} of two non-negative operators
$H_1,H_2$ with purely discrete spectrum is then defined as
\begin{equation*}
  \xi(H_1,H_2,\lambda) := n(H_2,\lambda) - n(H_1,\lambda).
\end{equation*}
We have the following estimate:

\begin{lemma}
  \label{lem:ssf}
  Let $(X,\Omega,\Prob,\ell)$ be a random length metric graph (as
  described in \Subsec{ranlengthmod}) and $\Lambda \subset X$ be a
  compact topological subgraph. Let $L_1, L_2$ be two vertex
  conditions differing in the vertex set $V_\diff \subset V(\Lambda)$
  only, and such that the operators $\lapl {(\Lambda,
    \ell_\omega),L_i}$ are non-negative. Let $0 \le q$ be a bounded
  measurable potential and $H_i = \lapl {(\Lambda,\ell_\omega),L_i} +
  q$. Then we have
  \begin{equation}
    \label{eq:ssfineq} 
    | \xi (H_1,H_2,\lambda) |
    \le 2 \sum_{v \in V_\diff} \deg v. 
  \end{equation}
  Moreover, if $\map \rho {\R_+} \R$ is a monotone
  function with $\rho' \in \Lp [1] {\R_+}$, then
  \begin{equation}
    \label{eq:ssftrineq} 
    \bigabs{\tr [\rho(H_1) - \rho(H_2)]}
    \le 2 | \rho(\infty) - \rho(0) | \sum_{v \in V_\diff} \deg v,
  \end{equation}
  where the trace is taken in the Hilbert space $\Lsqr
  {\Lambda,\ell_\omega}$.
\end{lemma}
\begin{proof} 
  Let ${\mathcal D}_0 = \dom H_1 \cap \dom H_2$. Then ${\mathcal D}_0$
  has finite index in $\dom H_i$, bounded above by twice the number of
  all edges adjacent to vertices $v \in V_\diff$. This implies $\dim
  (\dom H_i / \mc D_0) \le 2 \sum_{v \in V_\diff} \deg v$.
  Inequality~\eqref{eq:ssfineq} follows now from~\cite[Lemma
  9]{glv:07}. The second inequality~\eqref{eq:ssftrineq} follows
  readily from Krein's trace identity
  \begin{equation}
    | \tr \rho(H_1) - \rho(H_2) |
    \le \int_0^\infty \vert \rho'(\lambda) \vert
        \cdot | \xi(H_1,H_2,\lambda) | \dd \lambda.
  \end{equation}
\end{proof}

The following uniform resolvent boundedness holds in every random
length covering model:

\begin{lemma} \label{lem:trest} Let $(X,\Omega,\Prob,\ell,L,q)$ be a
  random length covering model with covering group $\Gamma$, as
  described in \Ass{covgraph}, and $\lambda > 0$. Then there is a
  constant $C_\lambda > 0$ such that we have
  \begin{equation*}
    \tr (H_\omega^\Lambda + \lambda)^{-1}
    \le C_\lambda\, \vol(\Lambda,\ell_0)
  \end{equation*}
  for all compact subgraphs $\Lambda \subset (X,\ell)$ and all $\omega
  \in \Omega$.
\end{lemma}

\begin{proof}
  Let $H_\omega^{\Lambda,0}$ denote the restriction on $\Lambda$ with
  Dirichlet vertex conditions \emph{at all vertices}. Then
  $H_\omega^{\Lambda,0} = \bigoplus_{e \in E(\Lambda)}
  H_\omega^{e,\Dir}$, where we identify the edge $e$ with the
  topological subgraph consisting of this edge and its end vertices in
  $X$. From~\eqref{eq:ssfineq} of \Lem{ssf} we conclude that
  \begin{equation*}
    | \tr \, (H_\omega^{\Lambda,0}+\lambda)^{-1} -
    (H_\omega^\Lambda+\lambda)^{-1} | \le \frac{4}{\lambda}\,
    |E(\Lambda)| = \frac{4}{\lambda} \vol(\Lambda,\ell_0).
  \end{equation*}
  Since $(H_\omega^{e,\Dir} +\lambda)^{-1}$ is bounded from above by
  $(\lapl [e,\Dir] {\omega} + \lambda)^{-1}$, and since the edges are
  uniformly bounded from above by $\ell_{\max}$, there is a constant
  $c_\lambda > 0$ such that $\tr\, (H_\omega^{e,\Dir} + \lambda)^{-1}
  \le c_\lambda$ for all $e \in E(\Lambda)$ and $\omega \in \Omega$.
  This implies the desired estimate with constant
  $C_\lambda=4\lambda^{-1}+ c_\lambda$.
\end{proof}

The proof of \Thm{ids} will now be given in four lemmata. All of these
lemmata are based on a given random length covering model
$(X,\Omega,\Prob,\ell,L,q)$ with an \emph{amenable} covering group
$\Gamma$ and a fixed \emph{tempered} F\o lner sequence $I_n$ with
associated compact topological graphs $\Lambda_n := \Lambda(I_n)$.

In the first lemma, we prove the convergence~\eqref{eq:idsmeasconv}
for a special family of functions $f_\lambda$ associated to resolvents
of the operators. Here, we need to apply an ergodic theorem of
Lindenstrauss~\cite{lindenstrauss:01}.

In later lemmata we show that the convergence \eqref{eq:idsmeasconv}
carries over to the uniform closure of finite linear combinations of
the functions $f_\lambda$, identify this closure with the help of the
Stone-Weierstrass Theorem, and finally conclude the desired
convergence for characteristic functions $\1_{[0,\lambda]}$ at
continuity points $\lambda > 0$ of the IDS.

\begin{lemma} \label{lem:Konvergenz-Resolventen} 
  Let $\lambda>0$ and $\map {f_\lambda} {[0 , \infty)} {\R}$, $f_\lambda (x) =
  \frac{1}{x+\lambda}$. Then there exists a subset $\Omega_0 \subset \Omega$
  of full $\Prob$-measure such that
  \begin{equation*}
    \lim_{n \to \infty} \frac 1 {\vol(\Lambda_n,\ell_\omega)}
      \tr [f_\lambda(H_\omega^{n,\Dir})]
    = \frac 1 {\Exp{\vol (\mc F,\ell_\bullet)}}
               \Exp{\tr [\1_{\mathcal F} f_\lambda(H_\bullet)]}
  \end{equation*}
  for all $\omega \in \Omega_0$.
\end{lemma}

\begin{proof} 
  We first consider a fixed $\omega \in \Omega$ and a fixed $\Lambda =
  \Lambda(I_n)$ and suppress the parameters $\omega$ and $n$ in the
  notation. Recall the definitions of $H^{\Lambda,\Dir}$ and
  $H^\Lambda$ with quadratic form domains given below. Let $\Lambda'$
  denote the closure of the complement $\Lambda^c$ in the metric graph
  $(X,\ell)$. By \Lem{brack.op} we have~\eqref{eq:brack.op} in the
  sense of quadratic forms. Since taking inverses is operator
  monotone, this implies
  \begin{equation*}
    (H^{\Lambda,\Dir} \oplus H^{\Lambda',\Dir} + \lambda)^{-1} 
    \le (H + \lambda)^{-1}
    \le (H^\Lambda \oplus H^{\Lambda'} + \lambda)^{-1}
  \end{equation*}
  for all $\lambda >0$.  In particular, we obtain inequalities for the
  following restricted quadratic forms: Set $(H +
  \lambda)^{-1}_\Lambda = p_\Lambda (H + \lambda)^{-1} i_\Lambda$,
  where $i_\Lambda$ and $p_\Lambda$ denote the canonical inclusions
  and projections between $\Lsqr {\Lambda,\ell} $ and $\Lsqr
  {X,\ell}$. Then
  \begin{equation}
    \label{eq:mon}
    (H^{\Lambda,\Dir} + \lambda)^{-1}
    \le (H + \lambda)_\Lambda^{-1}
    \le (H^\Lambda + \lambda)^{-1}.
  \end{equation}
  Consequently, $(H + \lambda)_\Lambda^{-1} - (H^{\Lambda,\Dir} +
  \lambda)^{-1}$ is non-negative and we have
  \begin{multline*}
    0 \le \tr_{\Lsqr {\Lambda,\ell}}
    \bigl[ 
       (H+\lambda)^{-1}_\Lambda -(H^{\Lambda,\Dir} +\lambda)^{-1}
    \bigr]\\
    \le \tr_{\Lsqr {\Lambda,\ell}}
    \bigl[
       f_\lambda(H^\Lambda) - f_\lambda(H^{\Lambda,\Dir})
    \bigr]
    \le \frac{2}{\lambda} d_{\max} |\bd \Lambda|,
  \end{multline*}
  using \Lem{ssf}, where $d_{\max}$ is a finite upper bound on the
  vertex degree of $X$, which exists due to the $\Gamma$-periodicity
  of $X$. Using the van Hove property~\eqref{eq:isop0} and the
  estimate
  \begin{equation*}
    \ell_{\min} \vol(\Lambda,\ell_0)
    \le \vol(\Lambda,\ell_\omega)
    \le \ell_{\max} \vol(\Lambda,\ell_0),
  \end{equation*}
  we conclude that
  \begin{equation} \label{eq:resdiff} 
  \lim_{n \to \infty} \frac{1}{\vol(\Lambda_n,\ell_\omega)} \left( \tr
    [(H_\omega+\lambda)^{-1}_{\Lambda_n}] - \tr [f_\lambda(H_\omega^{n,\Dir})]
  \right) = 0.
  \end{equation}
  Using additivity of the trace and the operator
  consistency~\eqref{eq:opcons}, we obtain
  \begin{equation*}
    \tr_{\Lsqr {\Lambda_n,\ell_\omega}}
    (H_\omega+\lambda)^{-1}_{\Lambda_n}
    = \sum_{\gamma \in I_n}
    \tr_{\Lsqr {\gamma \mc F,\ell_\omega}}
    (H_\omega+\lambda)^{-1}_{\gamma \mc F}
    = \sum_{\gamma \in I_n^{-1}} g_\lambda(\gamma \omega),
  \end{equation*}
  where
  \begin{equation}
    \label{eq:glambda}
    g_\lambda(\omega)
    = \tr_{\Lsqr {\mc F,\ell_\omega}}
       [(H_\omega+\lambda)^{-1}_{\mathcal F}]
    = \tr[\1_{\mathcal F} f_\lambda(H_\omega)].
  \end{equation}
  Since, by monotonicity~\eqref{eq:mon} and \Lem{trest},
  \begin{equation*}
    0 \le g_\lambda(\omega)
    \le \tr_{\Lsqr {\mc F,\ell_\omega}} 
           [(H_\omega^{\mathcal F}+\lambda)^{-1}]
    \le C_\lambda \vol(\mc F,\ell_0),
  \end{equation*}
  we conclude that $g_\lambda \in \Lp [1] \Omega$. Now, we argue as in
  the proof of Theorem 7 in~\cite{lpv:04}: Applying Lindenstrauss'
  ergodic theorem separately to both expressions
  \begin{equation*}
    \frac 1 {| I_n |} \sum_{\gamma \in I_n^{-1}} 
            g_\lambda(\gamma \omega)
            \qquad \text{and} \qquad
    \frac 1 {| I_n |} \sum_{\gamma \in I_n^{-1}} 
            \vol({\mathcal F},\ell_{\gamma \omega}),
  \end{equation*}
  we conclude that
  \begin{equation}
    \label{eq:limids}
    \lim_{n \to \infty} \frac{1}{\vol(\Lambda_n,\ell_\omega)} \tr
    [(H_\omega+\lambda)_{\Lambda_n}^{-1}]
    = \frac{1}{\Exp{\vol ({\mathcal
          F},\ell_\bullet)}} \Exp{\tr [\1_{\mathcal F}
      f_\lambda(H_\bullet)]}
  \end{equation}
  for almost all $\omega \in \Omega$.  The lemma follows now
  immediately from~\eqref{eq:resdiff} and~\eqref{eq:limids}.
\end{proof}

Let us denote by $\mathcal L$ the set of functions $\set{x \mapsto
  f_\lambda(x) = (x+\lambda)^{-1}}{\lambda > 0}$ and by $\mathcal A$
the $\norm[\infty]\cdot$-closure of the linear span of $\mathcal L$
and the constant function $\map \1 {[0,\infty)} \R$, $\1(x) = 1$. Note
that, by monotonicity~\eqref{eq:mon} and \Lem{trest}, both expressions
$\mu_\omega^n(f_1)$ and $\mu(f_1) = ( \Exp{ \vol({\mathcal
    F},\ell_\bullet) } )^{-1} \Exp {g_1}$ (with $g_1$ defined
in~\eqref{eq:glambda}) are bounded by a constant $K > 0$, independent
of $\omega$ and $n$. Let $\Omega_0 \subset \Omega$ be the set of full
$\Prob$-measure from \Lem{Konvergenz-Resolventen}.

\begin{lemma}
  \label{lem:Konvergenz-Abschluss}
  Let $\omega \in \Omega_0$. Set $ \nu^n = f_1 \cdot \mu_\omega^n$
  (for $n\in \N$) and $ \nu= f_1 \cdot \mu$. Then we have, for all $g
  \in {\mathcal A}$,
  \begin{equation*}
    \lim_{n\to\infty} \nu^n (g) = \nu(g).
  \end{equation*}
\end{lemma}

\begin{proof}
  By \Lem{Konvergenz-Resolventen} we know that the statement
  holds for the function $g = \1$. We note that $f_\lambda \cdot f_1 =
  \frac{1}{\lambda-1}(f_\lambda-f_1)$ for $\lambda \neq 1$. Thus, by
  linearity and \Lem{Konvergenz-Resolventen}, the
  convergence holds also for all functions $g = f_\lambda$ with
  $\lambda >0$, $\lambda \neq 1$. To deal with the case $\lambda=1$
  note that $f_{1 +\eps}$ converges to $f_1$ uniformly, as $\eps \to
  0$. Thus
  \begin{equation*}
    |\nu^n(f_1) - \nu^n(f_{1 +\eps})|
    \le \| f_1-f_{1 +\eps}\|_\infty \, \nu^n(\1)
    \le K \eps.
  \end{equation*}
  An analogous statement holds for $\nu^n$ replaced by $\nu$. Thus
  \begin{equation}
    \label{eq:uniform-approximand}
    |\nu(f_1)-\nu^n(f_1)| \le 2 K \eps + \big| \nu(f_{1 +\eps})
    - \nu^n(f_{1 +\eps}) \big| \to 2 K \eps,
  \end{equation}
  as $n\to\infty$. Since $\eps > 0$ was arbitrary, we conclude that
  $\lim_{n\to\infty} \nu^n (f_1) = \nu(f_1)$. By linearity, the
  convergence statement of the Lemma holds for all functions $g$ in
  the linear span of $\mathcal L \cup \{ \1 \}$. To show that is holds
  for all functions in the closure $\mathcal A$, as well, one uses
  uniform approximation and an estimate of the same type as
  in~\eqref{eq:uniform-approximand}.
\end{proof}

The next lemma identifies the space $\mathcal A$ explicitly:

\begin{lemma}
  The function space $\mathcal A$ coincides with the set of continuous
  functions on $[0,\infty)$ which converge at infinity.
\end{lemma}

\begin{proof}
  The statement of the lemma is equivalent to $\mathcal A=
  \Cont{[0,\infty]}$, where $[0,\infty]$ is the
  one-point-compactification of $[0,\infty)$. We want to apply the
  Stone-Weierstrass Theorem. Any $f_\lambda$ with $\lambda > 0$
  separates points and $\1$ is nowhere vanishing in $[0,\infty]$. By
  definition $\mathcal A$ is a linear space. To show that it is an
  algebra we use again the formula $f_{\lambda_1} \cdot f_{\lambda_2}
  =\frac{1}{\lambda_2 - \lambda_1}(f_{\lambda_1}-f_{\lambda_2})$,
  which shows that $f_{\lambda_1}\cdot f_{\lambda_2} \in {\mathcal A}$
  for $\lambda_1 \neq \lambda_2$. Since $\mathcal A$ is closed in the
  sup-norm, we can use an approximation as in the proof of the
  \Lem{Konvergenz-Abschluss} to show $f_\lambda^2 \in {\mathcal A}$. A
  similar argument shows that the product of two limit points $f,g$ of
  the linear span of ${\mathcal L} \cup \{ \1 \}$ is in $\mathcal A$.
 \end{proof}

 We have established the convergence $\mu_\omega^n(g) \to \mu(g)$ for
 all functions of the form $g \cdot f_1$ with $g \in {\mathcal A}$.
 The following lemma shows that this is sufficient to conclude the
 almost sure convergence $N_\omega^n (\lambda) \to N(\lambda)$ at
 continuity points $\lambda$, finishing the proof of \Thm{ids}. One
 has only to observe that every continuous function of compact support
 on $\R^+ = [0,\infty)$ can be written as $g \cdot f_1$, with an
 element $g \in {\mathcal A}$.

\begin{lemma}
  For $n \in \N$, let $\rho^n, \rho$ be locally finite measures on
  $\R^+$. Then
  \begin{equation*}
    \lim_{n\to\infty} \rho^n(g) = \rho(g)
  \end{equation*}
  for all continuous functions $g$ of compact support implies that
  \begin{equation*}
    \lim_{n\to\infty} \rho^n( [0,\lambda] )
    = \rho( [0,\lambda] )
  \end{equation*}
  for all $\lambda>0$ which are not atoms of $\rho$.
\end{lemma}

\begin{proof}
  The proof is standard. First note that locally finiteness of $\rho$
  implies 
  \begin{equation*}
    \lim_{\eps \to 0} \rho([\lambda-\eps,\lambda +\eps])
    = \rho(\{\lambda\}) = 0.
  \end{equation*}
  Now choose monotone functions $g_\eps^-, g_\eps^+ \in \Contc{\R^+}$
  satisfying
  \begin{equation*}
    \1_{[0,\lambda -\eps]} 
    \le g_\eps^-
    \le \1_{[0,\lambda]}
    \le g_\eps^+
    \le \1_{[0,\lambda +\eps]}.
  \end{equation*}
  Then
  \begin{align*}
    \rho([0,\lambda]) - \rho^n([0,\lambda]) &\le \rho(g_\eps^+) -
    \rho(g_\eps^-) + \rho(g_\eps^-) -
    \rho^n(g_\eps^-) \\
    &\le \rho([\lambda-\eps,\lambda +\eps]) + \rho(g_\eps^-) -
    \rho^n(g_\eps^-).
  \end{align*}
  For any $\delta >0$ one can choose $\eps >0$ such that
  $\rho([\lambda-\eps,\lambda +\eps]) < \delta$. Since $\delta>0$ was
  arbitrary, we have shown $\rho([0,\lambda]) \le \liminf_{n\to\infty}
  \rho^n([0,\lambda])$.  The opposite inequality is shown similarly.
\end{proof}

\section{Proof of the Wegner estimate}
\label{sec:wegner}

This section is devoted to the proof of \Thm{wegner}. Let
$(X,\Omega,\Prob,\ell)$ be a random length model satisfying
\Ass{wegner}. We first introduce a new measurable map $\map \alpha
{\Omega \times E} {[\omega_-,\omega_+]}$ with $\omega_- = \ln
\ell_{\min}$, $\omega_+ = \ln \ell_{\max}$, defined by
$\alpha_\omega(e) := \alpha(\omega,e) = \ln \ell_\omega(e)$. The
random variables $\alpha(\cdot,e)$, $e \in E$, are independently
distributed with density functions $g_e(x) = \e^x h_e(\e^x)$, and we
have
\begin{equation}
  \label{eq:tihediff}
  \norm[\infty] {g_e'} \le
  \ell_{\max} \norm[\infty]{h_e}
  + \ell_{\max}^2 \norm[\infty]{h_e'}
  \le (\ell_{\max} + \ell_{\max}^2) C_h
  =: D_h < \infty.
\end{equation}
Thus, we can re-identify $\Omega$ with the Cartesian product $\prod_{e
  \in E} [\omega_-,\omega_+]$, and the maps $\alpha(\cdot,e)$ are
simply projections to the component with index $e$. The measure
$\Prob$ is now given as the product $\bigotimes_{e \in E} \wt \Prob_e$
of marginal measures $\wt \Prob_e$ with density functions $g_e \in
\Cont[1] \R$ satisfying the above estimate~\eqref{eq:tihediff}. The
advantage of the new ``rescaled'' identification $\Omega = \prod_{e
  \in E} [\omega_-,\omega_+]$ is the following property of the
eigenvalues of the Laplacian on any compact subgraph
$(\Lambda,\ell_\omega)$:
\begin{equation}
  \label{eq:lambdai.scal}
  \lambda_i(\Delta_{\omega+s \1}^{\Lambda,\Dir})
  = \e^{-2 s}\lambda_i(\Delta_{\omega}^{\Lambda,\Dir}).
\end{equation}
Here, the eigenvalues $\lambda_i$ are counted with multiplicity and
$\omega + s \1$ denotes the element $\{\omega_e + s\}_{e \in
  E(X)} \in \Omega$.  Property~\eqref{eq:lambdai.scal} is an
immediate consequence of~\eqref{eq:D.resc}, $\ell_{\omega+s \1}(e) =
\e^{\alpha_\omega(e)+s} = \e^s \ell_\omega(e)$, and the fact that a
rescaling of all lengths by a fixed multiplicative constant does not
change the domain of the Kirchhoff Laplacian with Dirichlet boundary
conditions on $\bd \Lambda$. Property~\eqref{eq:lambdai.scal} is of
crucial importance for the proof of the Wegner estimate.

Henceforth, we use this new interpretation of $\Omega$ and rename
$\wt \Prob_e$ by $\Prob_e$, for simplicity. 

Let $\Lambda \subset X$ be a compact topological subgraph, $\lambda
\in \R$ and $\eps > 0$. We write the interval $I$ as
$[\lambda-\epsilon,\lambda+\epsilon]$ and start with a smooth function
$\map \rho \R {[-1,0]}$ satisfying $\rho \equiv -1$ on
$(-\infty,-\eps]$, $0 \le \rho' \le 1/\eps$, $\rho \equiv 0$ on
$[\eps,\infty)$.  Moreover, we set $\rho_\lambda(x) =
\rho(x-\lambda)$. Then we have
\begin{equation*}
  \1_{[\lambda-\eps,\lambda+\eps]}(x) \le
    \rho_\lambda(x+2\eps) -
    \rho_\lambda(x-2\eps) = \int_{-2\eps}^{2\eps} \rho_\lambda'(x+t) 
\dd t.
\end{equation*}
Using the spectral theorem, we obtain
\begin{equation*}
  P^{\Lambda,\Dir}_\omega([\lambda-\eps,\lambda+\eps])
  =
    \1_{[\lambda-\eps,\lambda+\eps]}(\Delta_\omega^{\Lambda,\Dir})
    \le \int_{-2\eps}^{2\eps}
        \rho_\lambda'(\Delta_\omega^{\Lambda,\Dir}+t) \dd t,
\end{equation*}
and, consequently,
\begin{equation*}
  \tr P^{\Lambda,\Dir}_\omega([\lambda-\eps,\lambda+\eps])
 \le
   \int_{-2\eps}^{2\eps} \tr
   \rho_\lambda'(\Delta_\omega^{\Lambda,\Dir}+t) \dd t.
 \end{equation*}
 Denote by $(\Omega(\Lambda),\Prob_\Lambda)$ the space
 $\Omega(\Lambda) = \prod_{e \in E(\Lambda)} [\omega_-,\omega_+]$ with
 probability measure $\Prob_\Lambda = \bigotimes_{e \in E(\Lambda)}
 \Prob_e$, and $\Exp[\Lambda] \cdot$ denote the associated
 expectation. $\Exp \cdot$ means expectation with respect to the full
 space $(\Omega,\Prob)$.  Applying expectation yields
\begin{multline}
   \label{eq:exptr} 
  \Exp {\tr
    P^{\Lambda,\Dir}_\bullet([\lambda-\eps,\lambda+\eps])}
  = \Exp[\Lambda] 
      {\tr P^{\Lambda,\Dir}_\bullet([\lambda-\eps,\lambda+\eps])}\\
  \le \int_{\Omega(\Lambda)}
  \int_{-2\eps}^{2\eps}
    \tr \rho_\lambda'(\Delta_\omega^{\Lambda,\Dir}+t) \dd t
     \dd \Prob_\Lambda(\omega).
\end{multline}
Using the chain rule and scaling property~\eqref{eq:lambdai.scal}, we obtain
\begin{multline*}
  \sum_{e \in E(\Lambda)} \frac{\partial}{\partial \omega_e}
  \rho_\lambda(\lambda_i(\Delta_\omega^{\Lambda,\Dir})+t)
  = \rho_\lambda'(\lambda_i(\Delta_\omega^{\lambda,\Dir})+t) \, 
       \frac d{ds}\Bigl\vert_{s = 0}\Bigr. 
       \bigl(
         s \mapsto \lambda_i(\Delta_{\omega+s \1}^{\Lambda,\Dir}) 
       \bigr)\\  
  = -2 \rho_\lambda'(\lambda_i(\Delta_\omega^{\lambda,\Dir})+t)
  \lambda_i(\Delta_{\omega}^{\Lambda,\Dir}) \le 0.
\end{multline*}
Now, we use that $[\lambda-\eps,\lambda+\eps] \subset J_u = [1/u,u]$.
Since $\supp \rho_\lambda' \subset [\lambda-\eps,\lambda+\eps]$, we
derive
\begin{equation} \label{eq:rhopest} 0 \le
  \tr \rho_\lambda'(\Delta_\omega^{\Lambda,\Dir}+t) \le -
  \frac u 2 \Bigl(\sum_{e \in E(\Lambda)} 
                    \frac{\partial}{\partial \omega_e}
                    \tr \rho_\lambda(\Delta_\omega^{\Lambda,\Dir}+t)
            \Bigr).
\end{equation}

For $e \in E(\Lambda)$, denote by $\Lambda_e$ the topological subgraph
with vertex set $V_e := V(\Lambda)$ and edge set $E_e := E(\Lambda)
\setminus \{ e \}$. Using the estimate~\eqref{eq:rhopest}, we obtain
from~\eqref{eq:exptr}
\begin{multline}
  \label{eq:expcalc}
  \Exp {\tr P_\bullet^{\Lambda,\Dir}([\lambda-\eps,\lambda+\eps])}\\
  \le - \frac{u}{2} \sum_{e\in E(\Lambda)} \int_{\Omega(\Lambda_e)}
    \int_{-2\eps}^{2\eps} \int_{\omega_-}^{\omega_+}
    \Bigl(
      \frac \partial {\partial \omega_e}
      \tr \rho_\lambda(\Delta_{(\omega',x)}^{\Lambda,\Dir} + t)
    \Bigr)
    g_e(x) \dd x \dd t \dd \Prob_{\Lambda_e}(\omega')
\end{multline}
with $(\omega',x) \in \Omega(\Lambda_e) \times [\omega_-,\omega_+] =
\Omega(\Lambda)$. Next, we want to carry out partial integration with respect
to $x$ in~\eqref{eq:expcalc}. Before doing so, it is useful to observe, for
fixed $c \in [\omega_-,\omega_+]$,
\begin{equation}
  \label{eq:lapldiff}
  \frac{\partial}{\partial \omega_e}
  \tr \rho_\lambda(\Delta_{(\omega',x)}^{\Lambda,\Dir} + t) =
  \frac{\partial}{\partial \omega_e} \left( 
    \tr \rho_\lambda(\Delta_{(\omega',x)}^{\Lambda,\Dir} + t) -
    \tr \rho_\lambda(\Delta_{(\omega',c)}^{\Lambda,\Dir} + t)
  \right).
\end{equation}
Using~\eqref{eq:lapldiff} and applying partial integration, we obtain
\begin{multline}
  \label{eq:expcalc2}
  \Bigabs{
    \int_{\omega_-}^{\omega_+} 
         \Bigl( \frac \partial {\partial \omega_e} 
           \tr \rho_\lambda(\Delta_{(\omega',x)}^{\Lambda,\Dir} + t)
         \Bigr) g_e(x) \dd x }\\
  \le \norm[\Lsymb^1]{g_e'} \sup_{c' \in [\omega_-,\omega_+]}
  \bigabs{\tr \rho_{\lambda-t}(\lapl[\Lambda,\Dir]{(\omega',c')}) - \tr
    \rho_{\lambda-t}(\lapl[\Lambda,\Dir]{(\omega',c)}) }.
\end{multline}

For notational convenience, we identify the compact topological graph
consisting only of the edge $e$ and its end-points with $e$, and we
denote by $\Delta_c^{e,\Dir}$ be the Dirichlet-Laplacian on the metric
graph $(e,\ell_c)$ defined by $\ell_c(e) = \exp(c)$.
Using~\eqref{eq:ssftrineq} in \Lem{ssf}, we conclude that
\begin{equation*}
  \bigabs{\tr \rho_{\lambda-t}(\lapl[\Lambda,\Dir]{(\omega',c)}) - 
      \tr \rho_{\lambda-t}(\lapl[\Lambda_e,\Dir]{\omega'} 
              \oplus \lapl[e,\Dir]{c})}
  \le  2\, | \rho(\infty) - \rho(t-\lambda) |\, 2 d_{\max} 
  \le 4 d_{\max},
\end{equation*}
for all values $c \in [\omega_-,\omega_+]$. Consequently, $\sup
\bigabs{\tr \rho_{\lambda-t}(\lapl[\Lambda,\Dir]{(\omega',c')}) - \tr
  \rho_{\lambda-t}(\lapl[\Lambda,\Dir]{(\omega',c)}) }$
in~\eqref{eq:expcalc2} can be estimated from above by
\begin{equation*}
  8 d_{\max} + \bigabs{\tr \rho(\lapl[e,\Dir]{c'} +t-\lambda) - \tr
  \rho(\lapl[e,\Dir]{c} +t-\lambda)}.
\end{equation*}
Note that all eigenfunctions of the Dirichlet operator $\lapl[e,\Dir]{c}$ are
explicitly given sine functions. Therefore, since $\lambda \in
[1/u+\eps,u-\eps]$ and $t \in [-2\eps,2\eps]$, there is a constant
$C_{u,\ell_{\max}} > 0$, depending only on $u,\ell_{\max}$, such that
\begin{equation*}
  \bigabs{\tr \rho(\lapl[e,\Dir]{c} +t-\lambda)} \le C_{u,\ell_{\max}},
\end{equation*}
for all $\exp(c) \in [\ell_{\min},\ell_{\max}]$. This implies
\begin{equation*}
  \Bigabs{
    \int_{\omega_-}^{\omega_+} \left( \frac{\partial}{\partial
        \omega_e} \tr \rho_\lambda(\Delta_{(\omega',x)}^{\Lambda,\Dir} +
      t) \right) g_e(x) \, dx } \le (8 d_{\max} + 2
  C_{u,\ell_{\max}})\, \norm[{\Lp[1]{[\omega_-,\omega_+]}}] {g_e'}.
\end{equation*}
Plugging this into inequality~\eqref{eq:expcalc}, we finally obtain
\begin{equation*}
  \Exp {\tr P_\bullet^{\Lambda,\Dir}([\lambda-\eps,\lambda+\eps])}
  \le u\, (4 d_{\max} + C_{u,\ell_{\max}})\, D_h\, \ln
  \frac{\ell_{\max}}{\ell_{\min}}\, 4 \eps\, \abs{E(\Lambda)},
\end{equation*}
finishing the proof of \Thm{wegner}.


\begin{thebibliography}{EKK{\etalchar{+}}08}

\bibitem[AS93]{adachi-sunada:93}
T.~Adachi and T.~Sunada, \emph{Density of states in spectral geometry},
  Comment. Math. Helv. \textbf{68} (1993), no.~3, 480--493. \MR{94k:58149}

\bibitem[AM93]{aizenman-molchanov:93}
Michael Aizenman and Stanislav Molchanov, \emph{Localization at large disorder
  and at extreme energies: an elementary derivation}, Comm. Math. Phys.
  \textbf{157} (1993), no.~2, 245--278. \MR{MR1244867 (95a:82052)}

\bibitem[ASW06a]{asw:06a}
M.~Aizenman, R.~Sims, and S.~Warzel, \emph{Absolutely continuous spectra of
  quantum tree graphs with weak disorder}, Comm. Math. Phys. \textbf{264}
  (2006), no.~2, 371--389. \MR{MR2215610 (2006m:34058)}

\bibitem[ASW06b]{asw:06d}
\bysame, \emph{Stability of the absolutely continuous spectrum of random
  {S}chr\"odinger operators on tree graphs}, Probab. Theory Related Fields
  \textbf{136} (2006), no.~3, 363--394.

\bibitem[BLT85]{BellissardLT-85} J.~Bellissard, R.~Lima, and
  D.~Testard, \emph{Almost periodic {S}chr\"odinger operators},
  Mathematics + physics. Vol.\ 1, pages 1--64, World Sci.\ Publishing,
  Singapore, 1985.

\bibitem[BK05]{BourgainK-05} J.~Bourgain and C.~E. Kenig, \emph{On
    localization in the continuous {A}nderson-{B}ernoulli model in
    higher dimension}, Invent. Math. \textbf{161} (2005), 389--426.

\bibitem[BGP08]{bgp:08}
J.~Br\"uning, V.~Geyler, and K.~Pankrashkin, \emph{Spectra of self-adjoint
  extensions and applications to solvable {S}chr\"odinger operators}, Rev.
  Math. Phys. \textbf{20} (2008), 1--70.

\bibitem[CL90]{carmona-lacroix:90}
R.~Carmona and J.~Lacroix, \emph{Spectral theory of random {S}chr\"odinger
  operators}, Probability and its Applications, Birkh\"auser Boston Inc.,
  Boston, MA, 1990. \MR{MR1102675 (92k:47143)}

\bibitem[Cat97]{cattaneo:97}
Carla Cattaneo, \emph{The spectrum of the continuous {L}aplacian on a graph},
  Monatsh. Math. \textbf{124} (1997), no.~3, 215--235. \MR{MR1476363
  (98j:35127)}

\bibitem[CCF{\etalchar{+}}86]{ChayesCFST-86}
J.~T. Chayes, L.~Chayes, J.~R. Franz, J.~P. Sethna, and S.~A. Trugman,
\emph{On the density of states for the quantum percolation problem},
J. Phys. A \textbf{19} (1986), L1173--L1177.

\bibitem[CH94]{combes-hislop:94}
J.-M. Combes and P.~D. Hislop, \emph{Localization for some continuous, random
  {H}amiltonians in $d$-dimensions}, J. Funct. Anal. \textbf{124} (1994),
  149--180. \MR{95g:82047}

\bibitem[CHK07]{chk:07}
J.~M. Combes, P.~D. Hislop, and F.~Klopp, \emph{An optimal {W}egner estimate
  and its application to the global continuity of the integrated density of
  states for random {S}chr\"odinger operators}, Duke Math. J. \textbf{140}
  (2007), no.~3, 469--498.

\bibitem[CHN01]{chn:01}
J.~M. Combes, P.~D. Hislop, and Shu Nakamura, \emph{The {$L\sp p$}-theory of
  the spectral shift function, the {W}egner estimate, and the integrated
  density of states for some random operators}, Comm. Math. Phys. \textbf{218}
  (2001), 113--130. \MR{2002e:82034}

\bibitem[EHS07]{ehs:07}
P.~Exner, M.~Helm, and P.~Stollmann, \emph{Localization on a quantum graph with
  a random potential on the edges}, Rev. Math. Phys. \textbf{19} (2007), no.~9,
  923--939. \MR{MR2355567}

\bibitem[EKK{\etalchar{+}}08]{ekkst:08}
P.~Exner, J.~P. Keating, P.~Kuchment, T.~Sunada, and A.~Teplayaev (eds.),
  \emph{Analysis on graphs and its applications}, Proc. Symp. Pure Math., vol.
  77, Providence, R.I., Amer. Math. Soc., 2008.

\bibitem[FHS06]{fhs:06}
R.~Froese, D.~Hasler, and W.~Spitzer, \emph{Absolutely continuous spectrum for
  the {A}nderson model on a tree: a geometric proof of {K}lein's theorem}, J.
  Funct. Anal. \textbf{230} (2006), no.~1, 184--221.

\bibitem[FS83]{froehlich-spencer:83}
J.~Fr{\"o}hlich and T.~Spencer, \emph{Absence of diffusion in the {A}nderson
  tight binding model for large disorder or low energy}, Comm. Math. Phys.
  \textbf{88} (1983), 151--184.

\bibitem[GHV08]{ghv.in:08}
M.~J. Gruber, M.~Helm, and I.~Veseli\'c, \emph{Optimal {W}egner estimates for
  random {S}chr\"odinger operators on metric graphs}, in~\cite{ekkst:08}
  (2008), 409--422.

\bibitem[GLV07]{glv:07}
M.~J. Gruber, D.~Lenz, and I.~Veseli{\'c}, \emph{Uniform existence of the
  integrated density of states for random {S}chr\"odinger operators on metric
  graphs over {$\Bbb Z\sp d$}}, J. Funct. Anal. \textbf{253} (2007), no.~2,
  515--533. \MR{MR2370087}

\bibitem[GLV08]{glv.in:08}
M.~J. Gruber, D.~Lenz, and I.~Veseli\'c, \emph{Uniform existence of the
  integrated density of states for combinatorial and metric graphs over
  $\mathbb{Z}^d$}, in~\cite{ekkst:08} (2008), 87--108.

\bibitem[GV08]{GruberV-08} M.~Gruber and I.~Veseli{\'c}, \emph{The
    modulus of continuity of the ids for random {Schr\"odinger}
    operators on metric graphs}, Random Oper.
    Stochastic Equations \textbf{16} (2008), 1--10.

\bibitem[Har00]{harmer:00}
M.~Harmer, \emph{Hermitian symplectic geometry and the factorization of the
  scattering matrix on graphs}, J. Phys. A \textbf{33} (2000), no.~49,
  9015--9032. \MR{MR1811226 (2001m:05165)}

\bibitem[HV07]{helm-veselic:07}
M.~Helm and I.~Veseli{\'c}, \emph{Linear {W}egner estimate for alloy-type
  {S}chr\"odinger operators on metric graphs}, J. Math. Phys. \textbf{48}
  (2007), no.~9, 092107, 7. \MR{MR2355077}

\bibitem[HK02]{hislop-klopp:02}
P.~D. Hislop and F.~Klopp, \emph{The integrated density of states for some
  random operators with nonsign definite potentials}, J. Funct. Anal.
  \textbf{195} (2002), no.~1, 12--47. \MR{1 934 351}

\bibitem[HP06]{hislop-post:pre06}
P.~Hislop and O.~Post, \emph{Exponential localization for radial random quantum
  trees}, Preprint \texttt{math-ph/0611022}, to appear in Waves in Random media
  (2006).

\bibitem[HKN{\etalchar{+}}06]{hknsv:06}
D.~Hundertmark, R.~Killip, S.~Nakamura, P.~Stollmann, and I.~Veseli\'c,
  \emph{Bounds on the spectral shift function and the density of states}, Comm.
  Math. Phys. \textbf{262} (2006), no.~2-3, 489--503.

\bibitem[HLMW01]{HupferLMW-01b} T.~Hupfer, H.~Leschke, P.~M{\"u}ller,
  and S.~Warzel, \emph{Existence and uniqueness of the integrated
    density of states for {S}chr\"odinger operators with magnetic
    fields and unbounded random potentials}, Rev. Math.
  Phys.\textbf{13} (2001), 1547--1581.

\bibitem[Kir89]{Kirsch-89a}
W.~Kirsch, \emph{Random {Schr\"odinger} operators},
H.~Holden and A.~Jensen (eds.), Schr\"odinger Operators,
  Lecture Notes in Physics, \textbf{345}, Springer, Berlin, 1989.

\bibitem[Kir96]{kirsch:96}
W.~Kirsch, \emph{Wegner estimates and {A}nderson localization for alloy-type
  potentials}, Math. Z. \textbf{221} (1996), 507--512. \MR{97d:82043}

\bibitem[Kir07]{kirsch:pre07}
\bysame, \emph{An invitation to {R}andom {S}ch\"odinger operators}, Preprint
  (2007).

\bibitem[KM82]{KirschM-82c}
W.~Kirsch and F.~Martinelli,
\emph{On the density of states of {Schr\"odinger} operators with a random
  potential}, J. Phys. A: Math. Gen. \textbf{15} (1982), 2139--2156.

\bibitem[KM07]{KirschM-07} W.~Kirsch and B.~Metzger, \emph{The
    integrated density of states for random {Schr\"odinger}
    operators}, Spectral Theory and Mathematical Physics, Proceedings
  of Symposia in Pure Mathematics, vol.~76, AMS, 2007, pp.~649--698.

\bibitem[KV02]{kirsch-veselic:02b}
W.~Kirsch and I.~Veseli\'c, \emph{Existence of the density of states for
  one-dimensional alloy-type potentials with small support}, Mathematical
  Results in Quantum Mechanics (Taxco, Mexico, 2001), Contemp. Math., vol. 307,
  Amer. Math. Soc., Providence, RI, 2002, pp.~171--176.

\bibitem[KLPS06]{klps:06}
S.~Klassert, D.~Lenz, N.~Peyerimhoff, and S.~Stollmann, \emph{Elliptic
  operators on planar graphs: unique continuation for eigenfunctions and
  nonpositive curvature}, Proc. Amer. Math. Soc. \textbf{134} (2006), no.~5,
  1549--1559 (electronic). \MR{MR2199204 (2006k:05055)}

\bibitem[Kle96]{klein:96}
Abel Klein, \emph{Spreading of wave packets in the {A}nderson model on the
  {B}ethe lattice}, Comm. Math. Phys. \textbf{177} (1996), no.~3, 755--773.
  \MR{MR1385084 (97j:82083)}

\bibitem[Kle98]{klein:98}
\bysame, \emph{Extended states in the {A}nderson model on the {B}ethe lattice},
  Adv. Math. \textbf{133} (1998), no.~1, 163--184. \MR{MR1492789 (98k:82096)}

\bibitem[Klo95]{klopp:95a}
F.~Klopp, \emph{Localization for some continuous random {S}chr\"odinger
  operators}, Comm. Math. Phys. \textbf{167} (1995), no.~3, 553--569.
  \MR{MR1316760 (95m:82080)}

\bibitem[KP08]{klopp-pankrashkin:08}
F.~Klopp and K.~Pankrashkin, \emph{Localization on quantum graphs with random
  vertex couplings}, J. Statist. Phys. \textbf{131} (2008), 561--673.

\bibitem[KP09]{klopp-pankrashkin:09}
F.~Klopp and K.~Pankrashkin, \emph{Localization on quantum graphs with random
  edge length}, Lett. Math. Phys., \textbf{87} (2009), 99--114.


\bibitem[KS99]{kostrykin-schrader:99}
V.~Kostrykin and R.~Schrader, \emph{Kirchhoff's rule for quantum wires}, J.
  Phys. A \textbf{32} (1999), no.~4, 595--630.

\bibitem[KV06]{kostrykin-veselic:06}
V.~Kostrykin and I.~Veseli\'c, \emph{On the {L}ipschitz continuity of the
  integrated density of states for sign-indefinite potentials}, Math. Z.
  \textbf{252} (2006), no.~2, 367--392.

\bibitem[KS87]{kotani-simon:87}
S.~Kotani and B.~Simon, \emph{Localization in general one-dimensional random
  systems. {II}. {C}ontinuum {S}chr\"odinger operators}, Comm. Math. Phys.
  \textbf{112} (1987), no.~1, 103--119. \MR{MR904140 (89d:81034)}

\bibitem[Kuc91]{kuchment:91}
P.~Kuchment, \emph{On the {F}loquet theory of periodic difference equations},
  Geometrical and algebraical aspects in several complex variables ({C}etraro,
  1989), Sem. Conf., vol.~8, EditEl, Rende, 1991, pp.~201--209. \MR{MR1222215
  (94k:39001)}

\bibitem[Kuc04]{kuchment:04}
\bysame, \emph{Quantum graphs: {I}. {S}ome basic structures}, Waves Random
  Media \textbf{14} (2004), S107--S128.

\bibitem[Kuc05]{kuchment:05}
\bysame, \emph{Quantum graphs: {II}. {S}ome spectral properties of quantum and
  combinatorial graphs}, J. Phys. A \textbf{38} (2005), no.~22, 4887--4900.
  \MR{MR2148631}

\bibitem[KP07]{kuchment-post:07}
P.~Kuchment and O.~Post, \emph{On the spectra of carbon nano-structures}, Comm.
  Math. Phys. \textbf{275} (2007), no.~3, 805--826.

\bibitem[Len99]{lenz:99}
D.~Lenz, \emph{Random operators and crossed products}, Math. Phys. Anal. Geom.
  \textbf{2} (1999), no.~2, 197--220. \MR{MR1733886 (2001b:46108)}

\bibitem[LMV08]{lmv:08}
D.~Lenz, P.~M\"uller, and I.~Veseli{\'c}.
\newblock Uniform existence of the integrated density of states for models on
  {$\ZZ^d$}.
\newblock {\em Positivity}, \textbf{12}(4):571--589, 2008.


\bibitem[LP08]{lledo-post:08b}
F.~Lled\'o and O.~Post, \emph{Eigenvalue bracketing for discrete and metric
  graphs}, J. Math. Anal. Appl. \textbf{348} (2008), no.~2, 806--833.

\bibitem[LPPV08]{lppv:08}
D.~Lenz, N.~Peyerimhoff, O.~Post, and I.~Veseli\'c, \emph{Continuity properties
  of the integrated density of states on manifolds}, Jpn. J. Math. \textbf{3}
  (2008), no.~1, 121--161. \MR{MR2390185}

\bibitem[LPV04]{lpv:04}
D.~Lenz, N.~Peyerimhoff, and I.~Veseli\'c, \emph{Integrated density of states
  for random metrics on manifolds}, Proc. London Math. Soc. (3) \textbf{88}
  (2004), no.~3, 733--752. \MR{MR2044055 (2005d:58046)}

\bibitem[LPV07]{LenzPV-07}
\bysame, \emph{Groupoids, von {N}eumann algebras and the integrated density of
  states}, Math. Phys. Anal. Geom. \textbf{10} (2007), no.~1, 1--41.
  \MR{MR2340531}

\bibitem[LV08]{LenzV}
D.~Lenz and I~Veseli{\'c}, \emph{Hamiltonians on discrete structures: jumps of
  the integrated density of states and uniform convergence.}, to appear in
  Math. Z., arXiv:0709.2836 (2008).

\bibitem[Lin01]{lindenstrauss:01}
E.~Lindenstrauss, \emph{Pointwise theorems for amenable groups}, Invent. Math.
  \textbf{146} (2001), no.~2, 259--295. \MR{2002h:37005}


\bibitem[MSY03]{msy:03}
V.~Mathai, Th. Schick, and S.~Yates, \emph{Approximating spectral invariants of
  {H}arper operators on graphs. {II}}, Proc. Amer. Math. Soc. \textbf{131}
  (2003), no.~6, 1917--1923 (electronic). \MR{MR1955281 (2003m:58046)}

\bibitem[MY02]{mathai-yates:02}
V.~Mathai and S.~Yates, \emph{Approximating spectral invariants of {H}arper
  operators on graphs}, J. Funct. Anal. \textbf{188} (2002), no.~1, 111--136.
  \MR{MR1878633 (2002k:47070)}

\bibitem[Mat93]{Matsumoto-93} H.~Matsumoto, \emph{On the
  integrated density of states for the {S}chr\"odinger operators with
  certain random electromagnetic potentials}, J. Math.
    Soc. Japan \textbf{45} (1993), 197--214.

\bibitem[Nic85]{nicaise:85} S.~Nicaise, \emph{Some results on
      spectral theory over networks, applied to nerve impulse
      transmission}, Lecture Notes in Math., vol. 1171, Springer,
    Berlin, 1985, pp.~532--541. \MR{MR839024 (87m:94042)}

\bibitem[Pas71]{Pastur-71} L.~A. Pastur, \emph{Selfaverageability of
    the number of states of the {S}chr\"odinger equation with a random
    potential}, Mat. Fiz. i Funkcional. Anal., \textbf{238} (1971),
  111--116.

\bibitem[PF92]{pastur-figotin:92}
L.~Pastur and A.~Figotin, \emph{Spectra of random and almost-periodic
  operators}, Grundlehren der Mathematischen Wissenschaften, vol. 297,
  Springer-Verlag, Berlin, 1992. \MR{MR1223779 (94h:47068)}

\bibitem[Pos08]{post:08a}
O.~Post, \emph{Equilateral quantum graphs and boundary triples},
  in~\cite{ekkst:08} (2008), 469--490.

\bibitem[PV02]{peyerimhoff-veselic:02}
N.~Peyerimhoff and I.~Veseli{\'c}, \emph{Integrated density of states for
  ergodic random {S}chr\"odinger operators on manifolds},
 Geom.\ Dedicata, \textbf{91} (2002), 117--135


\bibitem[{Sh}u79]{Shubin-79} M.~A. {Sh}ubin, \emph{Spectral theory and
    the index of elliptic operators with almost-periodic
    coefficients}, Russ. Math. Surveys, \textbf{34} (1979) 109--157.

\bibitem[Sto01]{stollmann:01}
P.~Stollmann, \emph{Caught by disorder: Bound states in random media}, Progress
  in Mathematical Physics, vol.~20, Birkh\"auser Verlag, Basel, 2001.

\bibitem[Ves02]{veselic:02a}
I.~Veseli{\'c}, \emph{Wegner estimate and the density of states of some
  indefinite alloy-type {S}chr\"odinger operators}, Lett. Math. Phys.
  \textbf{59} (2002), no.~3, 199--214. \MR{MR1904501 (2003g:81057)}

\bibitem[Ves05]{Veselic-05b}
\bysame, \emph{Spectral analysis of percolation {H}amiltonians}, Math. Ann.
  \textbf{331} (2005), no.~4, 841--865. \MR{MR2148799 (2006g:81058)}

\bibitem[Ves07]{veselic:07}
I.~Veseli{\'c}, \emph{Lifshitz asymptotics for {Hamiltonians} monotone in the
  randomness}, Oberwolfach Rep. \textbf{4} (2007), no.~1, 380--382.

\bibitem[Ves08]{veselic:pre08}
I.~Veseli\'c, \emph{{Wegner estimates for sign-changing single site
  potentials}}, arXiv:0806.0482 (2008).

\bibitem[vB85]{von-below:85}
J.~von Below, \emph{A characteristic equation associated to an eigenvalue
  problem on {$C\sp 2$}-networks}, Linear Algebra Appl. \textbf{71} (1985),
  309--325. \MR{MR813056 (87i:94030)}

\bibitem[Weg81]{Wegner-81}
F.~Wegner, 
\emph{Bounds on the {DOS} in disordered systems},
Z. Phys. B \textbf{44} (1981), 9--15.

\end{thebibliography}
\newcommand{\etalchar}[1]{$^{#1}$}
\providecommand{\bysame}{\leavevmode\hbox to3em{\hrulefill}\thinspace}
\renewcommand{\MR}[1]{}
\renewcommand{\MRhref}[2]{}
\providecommand{\href}[2]{#2}


\end{document}